\documentclass{article}
\usepackage{amsmath,amssymb}
\usepackage{graphicx}
\usepackage{fullpage}
\usepackage{cite}
\usepackage{palatino}
\usepackage{float}
\usepackage{subfig}
\usepackage{epstopdf}
\usepackage{tikz}
\usepackage{url}
\graphicspath{{../graphics/}}
\usepackage{soul}

\usepackage{amsmath,amsthm,amssymb}

\newcommand{\N}{\mathbb{N}}

\newcommand{\R}{\mathbb{R}}

\newtheorem{theorem}{Theorem}[section]
\newtheorem{lemma}[theorem]{Lemma}
\newtheorem{corollary}[theorem]{Corollary}
\newtheorem{proposition}[theorem]{Proposition}
\newtheorem{definition}[theorem]{Definition}
\newtheorem{remark}[theorem]{Remark}
\newtheorem{example}[theorem]{Example}

\renewcommand{\a}{\alpha}
\renewcommand{\b}{\beta}
\newcommand{\g}{\gamma}
\renewcommand{\d}{\delta}
\newcommand{\e}{\epsilon}

\newcommand{\1}{\mathbf{1}}
\newcommand{\0}{\mathbf{0}}
\newcommand{\av}[1]{\left|{#1}\right|}

\newcommand{\ip}[2]{\left\langle{{#1}},{{#2}}\right\rangle}

\renewcommand{\l}{\left}

\renewcommand{\r}{\right}

\newcommand{\be}{\begin{enumerate}}
\newcommand{\bi}{\begin{itemize}}
\newcommand{\ee}{\end{enumerate}}
\newcommand{\ei}{\end{itemize}}

\renewcommand{\P}{\mathbb{P}}
\newcommand{\E}{\mathbb{E}}

\newcommand{\tmu}{\widetilde{\mu}}
\newcommand{\Spec}{\mathrm{Spec}}

\newcommand{\tr}{\mathrm{Tr}}
\newcommand{\diag}{\mathrm{diag}}
\newcommand{\mc}[1]{\mathcal{{#1}}}

\newcommand{\df}{{{\mathsf{deg}}(f)}}
\newcommand{\dl}{{{\mathsf{deg}}(\lambda)}}
\newcommand{\SHS}{\mathsf{SHS}}

\title{Stochastic hybrid systems in equilibrium: moment closure,
  finite-time blowup, and exact solutions}
 
\author {Lee DeVille$^1$, Sairaj Dhople$^2$,\\
  Alejandro~D.~Dom\'{i}nguez-Garc\'{i}a$^{3}$, and Jiangmeng
  Zhang$^3$  \vspace{10pt} \\ $^1$ Dept. of Mathematics, University of
Illinois \\ $^2$ Dept. of Electrical and Computer Engineering,
University of Minnesota \\ $^3$ Dept. of Electrical and Computer
Engineering, University of Illinois }

\begin{document}
\maketitle

\begin{abstract}
  We present a variety of results analyzing the behavior of a class of
  stochastic processes --- referred to as Stochastic Hybrid Systems
  (SHSs) --- in or near equilibrium, and determine general conditions
  on when the moments of the process will, or will not, be
  well-behaved.  We also study the potential for finite-time blowups
  for these processes, and exhibit a set of random recurrence
  relations that govern the behavior for long times.  In addition, we
  present a connection between these recurrence relations and some
  classical expressions in number theory.
\end{abstract}

\section{Introduction}

\subsection{Background and motivation}

In this paper, we consider a class of stochastic processes referred to
as Stochastic Hybrid Systems (SHSs), and present results about their
behavior in or near equilibrium.

One motivation for considering SHSs is if we want to model a system
represented by a state that continuously evolves in time
(e.g. according to an ODE or an SDE), with the additional complication
that the variable, or even the state space itself, can undergo rapid
changes.  To consider one specific example, consider an engineered
system where the state of the system is given by a vector $x(t)\in
\R^d$ when the system is in ``normal'' operation, but when the system
is in an ``impaired'' state, only the first $d'<d$ variables of the
state are able to evolve, and the remainder stay fixed.  Moreover,
assume that the switching between normal and impaired operation happens
randomly, but in a manner that depends on the state $x$.  

The study of SHSs has a long history, going back at least
to~\cite{Bellman.54}; the main theoretical foundations of the field
were laid out in the book~\cite{Davis-1993}.  There is a large
literature devoted to the understanding the stability analysis of such
systems~\cite{Chatterjee.Liberzon.04, Chatterjee.Liberzon.06,
  Yin.Zhu.book} (this stability is typically understood in a moment or
almost-sure sense) and more recent work aimed at developing methods to
explicitly compute or estimate observables of such
systems~\cite{HespanhaSingh-2005, Hespanha-2005, Hespanha-2006}.
Additionally, metastability and large deviations in SHSs were studied
in~\cite{Bressloff.Newby.14} using path-integral methods.  SHSs
represent a powerful formalism that has been applied to many fields,
including: networked control
systems~\cite{AntunesHespanhaSilvestreFeb13}, power
systems~\cite{TCAS13}, system reliability theory~\cite{Dhople2014158},
and chemical reaction dynamics~\cite{SinghHespanhaFeb11}. A recent
review of the state of the art of such systems, that also contains an
extensive history, bibliography, and list of applications of SHSs, is
the comprehensive review~\cite{Teel.Subbaraman.Sferlazza.14}.

The state space of an SHS is comprised of a \textit{discrete state}
and a \textit{continuous state}; the pair formed by these is what we
refer to as the \textit{combined state} of the SHS. One can think of a
SHS as a family of stochastic differential equations (SDEs) for the
continuous state that are indexed by the discrete state. The discrete
state changes stochastically, and we think of this as switching
between SDEs.  Additionally, each discrete-state transition is
associated with a reset map that defines how the pre-transition
discrete and continuous states map into the post-transition discrete
and continuous states, so that the continuous variable can be
``reset'' when the discrete state changes.  We denote the discrete
state space by $\mc Q$, and assume that the continuous state space is
$\R^d$.  Then choose $\av{\mc Q}$ different SDEs, indexed by $q\in\mc
Q$, so that when the discrete state is $q$, $X_t$ evolves according to
\begin{equation*}
  dX_t = b_q(X_t) \,dt + \sigma_q(X_t)\,dW_t.
\end{equation*}
Then, we assume that there are a family of rate functions
$\lambda_{q,q'}(x)$ and reset functions $\psi_{q,q'}(x)$ such that if
the discrete state is $q$, the probability of a jump to state $q'$ in
the next $\Delta t$ is $\lambda_{q,q'}(X_t)\,\Delta t + o(\Delta t)$,
and, if such a jump occurs, the map $\psi_{q,q'}$ is applied to the
continuous state at the time of jump.  [This process can be described
more precisely in a ``non-asymptotic'' formalism, and we do this in
Section~\ref{sec:definition}.]  From this, it follows that there is a
continuous-time process, $Q_t$, that evolves according to some law,
and the full system we are studying is
\begin{equation*}
  dX_t = b_{Q_t}(X_t)\,dt + \sigma_{Q_t}(X_t)\,dW_t.
\end{equation*}

It is not hard to show that the process defined in this way is strong
Markov~\cite[\S 25]{Davis-1993}.  To fully characterize the SHS, we
need to compute the expectation of some large class of functions
evaluated on its state space. For the purposes of this paper, we want
to understand the process in equilibrium, or on the way to
equilibrium, i.e., the statistics of the process after it has evolved
for a long time.  The existence~\cite{Davis-1993, Meyn.Tweedie.book}
and smoothness~\cite{Bakhtin.Hurth.12} of the invariant measure of
these processes has been established, and in one sense, the goal of
this paper is to explicitly compute as much about this measure as we
can, and we approach this by studying the moments of the process,
i.e. the expectation of polynomial functions of the process in
equilibrium.

Using the infinitesimal generator of the process,
following~\cite{Hespanha-2006, Dhople2014158}, we can write down a set
of differential equations for the moments of this process --- we refer
to this set of equations as the {\em moment flow equations}.  These
equations are {\em a priori} infinite-dimensional, and as we show
below, in a wide variety of circumstances they are ``inherently
infinite-dimensional,'' by which we mean that: (i) there is no
projection of the dynamics onto a finite-dimensional subspace, and
(ii) any approximate projection into a finite-dimensional subspace
behaves poorly in a sense to be made precise below.  This inherent
infinite-dimensionality is typically called a {\em moment closure
  problem} in a variety of physical and mathematical contexts, in the
sense that one cannot ``close'' the moment flow equations in a finite
way.  Examples of moment closure problems and various approaches to
handle them span a wide range, including applications in: chemical
kinetics~\cite{Milner.Gillespie.Wilkinson.11}, dynamic
graphs~\cite{Gross.etal.06, Gross.Kevrekidis.08, Rogers.11},
physics~\cite{Fry.82}, population
dynamics/epidemiology~\cite{Keeling.book, Keeling.99,
  Newman.Ferdy.Quince.04, Wilkinson.Sharkey.14}, and nonlinear
PDE~\cite{Bronski.McLaughlin.00, Colliander.etal.10}.  Our approach is
related to and inspired by the Lyapunov moment stability theory that
has been well-developed for diffusions, especially those with small
noise, over the past few decades~\cite{Khasminskii.book, Baxendale.91,
  Baxendale.92, Namachchivaya.vanRoessel.01}.

We mention that a moment closure problem for SHSs was solved nicely
in~\cite{Hespanha-2005, Hespanha-2006} and related works, but in
contrast the problem there was to find a moment closure approximation
that was valid for a finite time interval.  Since we are interested in
equilibrium or near-equilibrium statistics of the problem, we want to
study the problem on the infinite time interval.

\subsection{Overview of the results of this paper}\label{sec:overview}

In this paper, we assume that the dynamics on the continuous state are
deterministic, i.e., all of the SDEs are assumed to be ODEs, but there
is still randomness in how the discrete states switch amongst each
other.  [These types of systems are also called ``randomly switched
systems'' in the literature~\cite{Chatterjee.Liberzon.04}, but are a
subclass of the SHS formalism so it makes sense to have a unifying
name.]  The fact that the continuous evolution does not have a
diffusion term makes the analysis both easier and harder {\em a
  priori} in the sense that the determinism allows for some
simplification in certain recurrence relations, but without
diffusions, the generator is only hypoelliptic, and thus the existence
of, and convergence to, invariant measures for the process is more
difficult to establish.  We discuss all of these issues below.

In this work, we want to understand the scenario where the two parts
of the system (the ODE and the reset) have competing effects.  For
example, if the ODE and the reset both send all orbits to zero, then
their mixture does as well, and, conversely, if they send all orbits
to infinity, then the mixture does as well.  What is more interesting
is the case in which the ODE and the resets have opposite effects, and we
consider one such case here: the ODE sends orbits to infinity, but the
resets send orbits toward zero.

We give a few prototypical examples of such systems, and discuss the
results of this paper that apply to each.  In each of these cases, we
present a model the continuous variable of which is defined on the
positive real line; this can be thought of as the fundamental model of
interest, or the radial component of a system in higher dimension (see
Section~\ref{sec:radial} for a discussion about dynamics of
higher-dimensional systems).

{\bf Example 1.}  Let $\alpha,\beta >0$ and $\gamma \in [0,1]$ and
define the ODE, jump rate, and reset as
\begin{equation}\label{eq:111}
  \dot x = f(x) = \alpha x,\quad \lambda(x) = \beta x,\quad x\mapsto \gamma x.
\end{equation}
The theory presented in Section~\ref{sec:mc1} implies that the
SHS in~\eqref{eq:111} converges to an invariant distribution with all
finite moments of all orders, and this fact is independent of the
values of $\alpha,\beta,\gamma$ (although of course the moments
themselves depend on these parameters).  

{\bf Example 2.}  Next, consider the system
\begin{equation}\label{eq:211}
  \dot x = f(x) = \alpha x^2,\quad \lambda(x) = \beta x,\quad x\mapsto \gamma x.
\end{equation}
 The theory of Section~\ref{sec:blowup} tells us that, depending on
 the value of $\gamma$, this process can have multiple behaviors: if
 $\gamma > e^{-\alpha/\beta}$, then all solutions of~\eqref{eq:211}
 blow up in finite time, if $1-\alpha/\beta < \gamma <
 e^{-\alpha/\beta}$, then solutions go to zero with high probability,
 but enough of them escape to infinite fast enough that the moments of
 the solution blow up in finite time, but if $\gamma <
 1-\alpha/\beta$, then the solution converges to zero in every sense.
 (See Theorem~\ref{thm:type1} for the specific statements.)

{\bf Example 3.}  Finally, consider the system
\begin{equation}\label{eq:311}
  \dot x = f(x) = \alpha x^3,\quad \lambda(x) = \beta x,\quad x\mapsto \gamma x.
\end{equation}
Theorem~\ref{thm:type2} tells us that all solutions of this system
blowup with probability one, and this fact is independent of the
parameters $\alpha,\beta,\gamma$.

\begin{remark}
  The theory in Sections~\ref{sec:mc1} and~\ref{sec:blowup}
  generalizes the statements in Examples~1 -- 3 to arbitrary
  polynomials.  The critical difference between the three cases are
  the relative degrees, and in some cases the leading-order
  coefficients, of $f$ and $\lambda$.  Additionally, although the two
  latter cases both exhibit finite-time blowup, they are of a much
  different character.  In the quadratic case, the system blows up in
  finite time even though there are infinitely many jumps, and as
  such, exhibit a quality very much like that of an explosive Markov
  chain.  In the cubic case, the blowup is more like that seen in
  nonlinear ODEs: the system goes off to infinity in finite time and
  there are only a finite number of jumps.
\end{remark}

In Examples~1~--~3, there was a single discrete state and all jumps
mapped the state back to itself.  We also analyze the moments of the
SHS with multiple discrete states in Section~\ref{sec:mc2}.  We show
that the behavior can be characterized similarly to that of one state
in many situations, but we also exhibit new types of behavior here.
For example, we show that in many cases, the system can exhibit {\em
  marginal moment stability}, by which we mean that in equilibrium,
the system can have some moments finite, and others that are infinite;
thus the equilibrium distribution has fat tails.  In this case, we
show numerically that the equilibrium distributions have power law
behavior.

We also show that the moment flow equations have a very
strange property.  These equations are an infinite-dimensional linear
system that supports a fixed point that corresponds to the moments of
the invariant measure.  We show that any finite-dimensional truncation
of this system has only unstable fixed points, and, moreover, as the
size of truncation grows, the system has larger positive eigenvalues.
This motivates the result that the infinite-dimensional system is
ill-posed in ``the PDE sense'' in a manner analogous to a
time-reversed heat equation.  However, we then show that the minimal
amount of convexity given by Jensen's Inequality is enough to make
this system well-posed and well-behaved, and in particular it then
becomes faithful to the stochastic process.  This seems to be a
strange example of a system that is ill-behaved on a linear space
becoming well-behaved when restricted to a nonlinear submanifold of
that same space.

\subsection{Organization of manuscript}

The main results and structure of this paper are as follows.  In
Section~\ref{sec:problem}, we give a formal definition of the SHS that
we study, and define the moment flow.  Next, we identify a broad class
of assumptions for SHSs under which the moment equations are
well-defined and accurate on the infinite-time interval, but we also
show that there is a surprising subtlety that arises in the
consideration of same; these results are contained in
Sections~\ref{sec:mc1} and~\ref{sec:mc2}. Next, in
Section~\ref{sec:blowup}, we study the SHS where moment closure fails,
and in fact the SHS undergoes finite-time blowups, i.e., the process
becomes infinite in expectation, or almost surely, at a finite time.
Finally, in Section~\ref{sec:exact}, we write down a few exact
recurrence relations for SHSs and show that we can use these, at least
in some cases, to compute arbitrarily good approximations.

\section{Problem formulation}\label{sec:problem}
In this section, we provide the formal definition of the SHS, and
define its generator. Then, we use the generator together with
Dynkin's formula to develop a set of differential equations that
describe the dynamics of the moments of the SHS. Finally, we provide a
description of the SHS using radial components, which allows us to
introduce a simplified model of the SHS; analyzing this model is
essentially the focus of the remainder of the paper.

\subsection{Definition of SHS}\label{sec:definition}

Consider a countable set $\mc{Q}$ of discrete states and a family of
phase spaces $P_q$, one for each discrete state.  We assume that we
have defined a family of (random) dynamical systems $D = (D_q)$, with
$D_q$ defined on phase space $P_q$.  The main notion driving SHS is
that we want to consider a stochastic process that: (i) in each small
timestep $dt$, the system can jump from state $q$ to state $q'$ with
probability $\lambda_{q'}(q,x)\, dt$; and (ii) if it does not jump, it stays in
state $q$ and evolves under the dynamics given by $D_q$.

As such, we are combining the standard models of a dynamical system on
a state space with the notion of a discrete-state Markov process; the
system jumps amongst a countable family of states as it would {be} a Markov
chain, and whenever it is in a particular state, it evolves according
to the dynamics attached to that state space.  Solving the model is
complicated by the fact that we allow the jump rates to depend on the
state, so that we cannot (for example) determine the jump times and
then solve for the continuous flows; they are intimately connected.

For the purposes of this paper, we will assume that the dynamical
systems on each phase space are ordinary differential equations
(ODEs), but there is no significant obstruction to generalizing the
individual flows to SDEs or even general random dynamical
systems~\cite{Yin.Zhu.book}.  We now give the precise definition of a
SHS.

\newcommand{\od}{\Omega_{\mathsf{d}}}
\newcommand{\fd}{\mc{F}_{\mathsf{d}}}
\newcommand{\fc}{\mc{F}_{\mathsf{c}}}
\newcommand{\oc}{\Omega_{\mathsf{c}}}

\begin{definition} \label{shs_def}
  A {\bf stochastic hybrid system (SHS)} is a quintuple $(\mc{Q},
  \mc{P}, \Lambda, \Psi, F)$ where
  \begin{itemize}
    
  \item $\mc{Q}$ is a countable set of {\bf discrete states};
    
  \item $\mc{P} = (P_q)_{q\in\mc{Q}}$ are a collection of manifolds,
    where $P_q$ is called the {\bf $q$th continuous state space};
    
  \item $\Lambda(q,x) = (\lambda_{q'}(q,x))_{q'\in\mc{Q}}$ is a
    collection of {\bf rate functions}; the domain of $\Lambda$ is
    $\{(q,x): x\in P_q, q \in \mc{Q}\}$ with the property that $\Lambda(q,x) \ge 0$
    on its domain.  We interpret $\lambda_{q'}(q,x)$ as the rate of
    jumping to discrete state $q'$ when the state is currently $(q,x)$.

  \item $\Psi(q,x) = (\psi_{q'}(q,x))_{q'\in\mc{Q}}$ is a collection
    of {\bf reset maps}; the domain of $\Psi$ is the same as
    $\Lambda$, and we assume that $\psi_{q'}(q,x)\in P_{q'}$.  We
    interpret $\psi_{q'}(q,x)$ as the new value of the continuous
    state immediately after making a $q\mapsto q'$ transition.
    
  \item $f(q,x)$ is a collection of vector fields, each defined on
    $P_q$, i.e., $f(q,\cdot)\colon P_q\to \mc{T}P_q$, where $\mc{T}P_q$ denotes the tangent space to  $P_q$.

\end{itemize}

We will denote the {\bf flow map generated by $f(q,\cdot)$} by
$\varphi^t(q,\cdot)$, i.e.,  
\begin{equation*}
\varphi^0(q,x) = x,\quad 
\varphi^s(q,\varphi^t(q,x)) = \varphi^{s+t}(q,x), \quad \mbox{and }
  \frac{d}{dt} \varphi^t(q,x) = f(q,\varphi^t(q,x)).
\end{equation*}

The {\bf continuous state space} is the disjoint union
$\coprod_{q\in\mc{Q}} P_q$ and we call $x\in\coprod_{q\in\mc{Q}} P_q$
a continuous state.  When all of the $P_q$ are the same, we abuse
notation and consider one copy of $P_q$ to be the entire continuous
state space.

We define the {\bf state} of the process as the pair $(Q_t,X_t)$, and
describe its evolution as follows.  Let us first assume that
$(Q_t,X_t)$ is known a.s. for some $t>0$.  Let $S^{(q)}, q\in \mc{Q}$
be iid exponential random variables with parameter~1, i.e.,
$\P(S^{(q)} > z) = e^{-z}$ for all $z\ge 0$, that are also independent
of $\{(Q_s,X_s):s \le t\}$.  Define stopping times $T^{(q)}, q\in\mc Q$
and $T$ as follows:
\begin{equation}\label{eq:defofT}
  \int_t^{T^{(q)}} \lambda_{q}(Q_t,\varphi^{s-t}(Q_t,X_t))\,ds = S^{(q)}, \quad 
  T = \inf_{q\in\mc Q} T^{(q)}.
\end{equation}
Specifically, the $T^{(q)}$ are the times at which each of the
transitions to state $q$ would fire, and we take the first one to do
so and ignore the others.  The time $T$ will be the time of the {\bf
  next jump}. We prescribe that the discrete state remains unchanged
until the next jump, and the continuous state flows according to the
appropriate ODE, i.e.
\begin{equation}\label{eq:path}
  Q_s = Q_t, \quad X_s = \varphi^{s-t}(Q_t,X_t), \mbox{ for all } s\in[t,T).
\end{equation}
Finally, we apply the appropriate reset map:
\begin{equation}\label{eq:reset}
  Q_{T} = \arg\inf_{q\in\mc{Q}} T_{q},\quad X_{T} = \psi_{Q_{T}}(Q_t, X_{T-}),
\end{equation}
where here and below we define
\begin{equation*}
  X_{T-} = \lim_{t\nearrow T}X_t.
\end{equation*}

Now that we know $(Q_T,X_T)$, we can define the process
recursively. More specifically, let us define $T_0 = 0$ and assume
$(Q_0,X_0)$ is known.  Then, denote by $T_1$ the time returned by the 
algorithm  above, and we have defined $X_t$ for $t\in[0,T_1]$.  For any
$n\in \N$, if we know $T_n$ and $(Q_{T_n},X_{T_n})$, then choose
$S_{n+1}^{(q)}$ iid exponential, and define $T_{n+1}$ by

\begin{equation}\label{eq:defofTn}
  \int_t^{T_{n+1}^{(q)}} \lambda_{q}(Q_t,\varphi^{s-t}(Q_t,X_t))\,ds = S_{n+1}^{(q)}, \quad 
  T_{n+1} = \inf_{q\in\mc Q} T_{n+1}^{(q)}.
\end{equation}
Then define $Q_t,X_t$ for $t\in[T_n,T_{n+1}]$ as
in~(\ref{eq:path},~\ref{eq:reset}).

We also use the convention throughout of {\bf minimality}: if
{$T_\infty:=\lim_{n\to\infty} T_n <\infty$}, then we say the process
{\bf explodes} or {\bf blows up} at time $T_\infty$, and set $X_t =
\infty$ for all $t>T_\infty$.  Similarly, or if there is a $T^*$ with
$\lim_{t\nearrow T^*} X_t = \infty$, then we say the process {\bf
  explodes} or {\bf blows up} at time $T_*$, and set $X_t = \infty$
for all $t>T_*$.

We will also write $(X_t,Q_t) = \SHS(\mc{Q}, \mc{P}, \Lambda, \Psi,
\varphi)$ to mean that $(X_t,Q_t)$ is a realization of the stochastic
process constructed using the procedure above.
\end{definition}

  This definition is complicated and we want to connect the formal
  definition of the process to an intuitive notion of what it should
  do. 

  First, note that it is clear from~\eqref{eq:path} that between jumps,
  the continuous state evolves according to the appropriate ODE, and
  the discrete state remains unchanged.  Moreover,
  from~\eqref{eq:reset} we see that if the system jumps, and it jumps
  from state $q$ to state $q'$, then $Q_T$ is updated to the value
  $q'$, and the continuous state is updated by applying the map
  $\psi_{q'}(q,X_{T-})$.

  It is also clear from the definition that the process is cadlag
  (i.e., is right-continuous and has left-limits at every jump, and
  continuous otherwise).  Moreover, the process is stationary, since
  the ODEs and the rate functions do not depend explicitly on time.

  Finally, we intuitively want that, given the current state $(Q_t,
  X_t)$, the probability of a jump occurring in $[t,t+\Delta t)$
  should be the sum of all of the current jump rates, i.e., $\sum_q
  \lambda_q(Q_t,X_t)\Delta t$, and the probability of jumping to state
  $q'$ is equal to the relative proportion of $\lambda_{q'}(Q_t,X_t)$
  to this total rate(qv.~Proposition~\ref{prop:intuition} below).

\begin{example}
  Before we state and prove Proposition~\ref{prop:intuition}, let us
  first consider the special case that many readers will be familiar
  with, namely, the case where $\lambda_{q'}(q,x)$ does not depend on
  $x$ at all and can be written $\lambda_{q'}(q)$.  Then the jumps are
  given by a continuous-time Markov chain (CTMC), and the
  $\av{\mc{Q}}^2$ numbers $\lambda_{q'}(q)$ are then the transition
  rates $q\mapsto {q'}$.  Start at $t=0$, and we obtain
  \begin{equation*}
    \int_0^{T^{(q)}} \lambda_q(Q_0)\,ds = S^{(q)},\mbox{ or, }T^{(q)} = S^{(q)}/\lambda_q(Q_0).
  \end{equation*}
  It is straightforward to see that the distribution of $T^{(q)}$ is
  then an exponential random variable of rate $\lambda_q(Q_0)$, i.e.
  \begin{equation*}
    \P(T^{(q)} > t) = e^{-\lambda_q(Q_0) t}.
  \end{equation*}
  Moreover, if each
  $T^{(q)}$ is exponential with rate $\lambda_q(Q_0)$, and $T = \inf_q
  T^{(q)}$, then $T$ is exponential of rate $\sum_q \lambda_q(Q_0)$  (see, e.g.~\cite[Theorem~2.3.3]{Norris.book}).
  Finally, $\P(\arg\inf_q T_q = k) = \lambda_k(Q_0) / \sum_q
  \lambda_q(Q_0)$, and this is independent of $T$.  In words, the rate
  of jumping to state $q$ is the constant rate $\lambda_q(Q_0)$, the
  total rate of any jump occurring is the sum of the individual rates
  of each jump occurring, and the probability of any given jump
  occurring is equal to its proportion to the sum of all the rates.

  The only way in which the current framework is more complicated than
  a standard CTMC is that the transition rates change in time, and
  they change in such a way as to make them a function of the
  continuous state.

\end{example}

\begin{proposition}\label{prop:intuition}
  We consider the SHS in Definition~\ref{shs_def} Start the SHS in state
  $(Q_0,X_0)$, and define $T$ as the time of the first jump.  Then for
  $t<T$, the discrete state $Q_0$ does not change, and $X_t$ evolves
  according to the ODE, so $X_t = \varphi^t(Q_0,X_0)$.  Then
  \begin{align*}
    \P(T \le t + \Delta t | T > t) &= \Delta t\cdot \sum_q \lambda_q(Q_0,X_t)=\Delta t\cdot\sum_q \lambda_q(Q_0,\varphi^t(Q_0,X_0)),\\
    \P(Q_{t+\Delta t} = q) &= \frac{\lambda_q(Q_0,X_t)}{\sum_{q'} \lambda_{q'}(Q_0,X_t)}.
  \end{align*}
\end{proposition}

\begin{proof}
  First note that $T > t$ iff $T^{(q)} > t$ for all $q$, and $T\le
  t+\Delta t$ iff $T^{(q)} \le t+\Delta t$ for some $q$.
  Next, we see that
\begin{align*}
  T^{(q)} > t &\Leftrightarrow \int_0^t \lambda_q(Q_0,X_s)\,ds < S^{(q)},\\
  T^{(q)} \le t + \Delta t &\Leftrightarrow \int_0^{t+\Delta t} \lambda_q(Q_0,X_s)\,ds \ge S^{(q)},
\end{align*}
and so
\begin{align*}
  \P(T^{(q)} \le t+\Delta t | T^{(q)} > t)
  &= \P\left(S^{(q)} \le \int_0^{t+\Delta t}\lambda_q(Q_0,X_s)\,ds \Bigg| S^{(q)} > \int_0^{t}\lambda_q(Q_0,X_s)\,ds\right)\\
  &= \P\left(S^{(q)}\le\int_0^{t+\Delta t}\lambda_q(Q_0,X_s)\,ds - \int_0^{t}\lambda_q(Q_0,X_s)\,ds\right),
\end{align*}
by the memorylessness property of exponential random variables.  But
note that for $\Delta t \ll 1$,
\begin{equation*}
  \int_0^{t+\Delta t}\lambda_q(Q_0,X_s)\,ds - \int_0^{t}\lambda_q(Q_0,X_s)\,ds = \Delta t \cdot \lambda_q(Q_0,X_t) + o(\Delta t),
\end{equation*}
and 
\begin{equation*}
  \P(S^{(q)} \le \Delta t \cdot \lambda_q(Q_0,X_t) + o(\Delta t)) = 1- \exp(-\Delta t \cdot \lambda_q(Q_0,X_t) + o(\Delta t)) = \Delta t \cdot \lambda_q(Q_0,X_t) + o(\Delta t).
\end{equation*}

In short, the probability that a jump to $q$ occurs in the next
$\Delta t$ is $\Delta t \cdot \lambda_q(Q_0,X_t)$ when $\Delta t$ is
sufficiently small, so the probability that some jump occurs is
clearly the sum of these (up to $o(\Delta t)$), and the relative
probability of it being $T_q$ is the relative proportion of the $q$th
rate with respect to the others.
\end{proof}

\begin{remark}
  Although the definition is given recursively for simplicity, it is
  not hard to see that we can choose all of the randomness of this
  process ``up front'', i.e., we can choose streams of iid
  exponentials $S_{n}^{(q)}$ for $n\in \N, q\in\mc Q$ at the
  beginning, then use these streams to determine the $n$th jump time
  $T_n$ (note that we are discarding all of the {$S_{n}^{(q')}$} that
  are not used in this step).  This allows us to define a map from any
  probability space $\Omega$ rich enough to contain the streams
  $S_{n}^{(q')}$ to $\mc{D}([0,\infty), \mc Q\times \mc P)$, the set
  of cadlag paths defined on $\mc Q\times \mathcal{P}$.  This induces
  a measure on the set of all paths, and, in particular, allows us to
  measure the probability of any event that can be determined by
  observing the paths.  When we talk about probabilities of events
  below, we are implicitly assuming that this correspondence has been
  made.  Moreover, this will even allow us to compare paths of two (or
  more) different SHS that are generated by different functions; as
  long as we have a correspondence from $\omega\in\Omega$ to the
  exponential streams $S_{n}^{(q)}(\omega)$, the paths are completely
  determined by~\eqref{eq:defofT}.  We will use this formalism
  throughout the remainder of the paper without further comment, and
  in general drop the dependence on $\omega$.
\end{remark}

\subsection{Infinitesimal generator}\label{sec:generator}

We follow the standard definition of the infinitesimal generator and
derive the generator of the process here.  Let $h\colon \mc{Q}\times
\mc{P}\to \R$ be bounded and differentiable, and   define the following linear
operator $\mc{L}$:
\begin{equation}\label{eq:defofL}
  \mc{L}h(q,x) := \lim_{\e\searrow0} \frac{\E[h(Q_{t+\e},X_{t+\e})|Q_t=q,X_t=x] - h(q,x)}{\e}.
\end{equation}
From this, we can directly  obtain {\em Dynkin's formula}:
\begin{equation}\label{eq:Dynkin}
  \frac{d}{dt}\E[h(Q_t,X_t)] = \E[\mc{L}h(Q_t,X_t)],
\end{equation}
or, said another way, the operator
\begin{equation}\label{eq:martingale}
  \mc{M}h(Q_t,X_t) = h(Q_0,X_0) + \int_0^t \mc{L}h(Q_s,X_s)\,ds
\end{equation}
is a martingale.  This allows us to extend the definition of the
domain of $\mc{L}$ through the martingale equation, and it is not hard
to show that under weak assumptions on $\Lambda$, $\Psi$, and {$f$},
the domain of $\mc{L}$ contains all polynomials~\cite{Davis-1993}.  In
particular, we can compute directly that
\begin{equation}\label{eq:generator2}
  \mc{L}h(q,x) = f(q,x)\cdot \nabla_x h(q,x) + \sum_{q'\in\mc{Q}} \lambda_{q'}(q,x)\left(h(q',\psi_{q'}(q,x)) - h(q,x)\right).
\end{equation}
To give a formal derivation of~\eqref{eq:generator2}, we compute
\begin{align*}
  &\E[h(Q_{t+\epsilon},X_{t+\epsilon}) - h(Q_t,X_t)|Q_t=q,X_t=x]\\
  &\quad= \E[h(Q_{t+\epsilon},X_{t+\epsilon}) - h(Q_t,X_t)|Q_t=q,X_t=x,\mbox{ no jump}]\P(\mbox{no jump}) \\&\quad\quad+ \sum_{q'\in\mc{Q}}\E[h(Q_{t+\epsilon},X_{t+\epsilon}) - h(Q_t,X_t)|Q_t=q,X_t=x,\mbox{ jump to $q'$}]\P(\mbox{ jump to $q'$})\\
  &= (h(q,x+\e f(q,x)) - h(q,x))\left(1-O(\epsilon)\right)+ \sum_{q'\in\mc{Q}} (h(q',\psi_{q'}(q,x))-h(q,x))(\epsilon \lambda_{q'}(q,x)) + O(\epsilon^2)\\
  &= \epsilon( f(q,x)\cdot \nabla_xh(q,x) + \sum_{q'\in\mc{Q}} \lambda_{q'}(q,x)\left(h(q',\psi_{q'}(q,x)) - h(q,x)\right)) + O(\epsilon^2),
\end{align*}
and from this and the definition of $\mc{L}$ we have
established~\eqref{eq:generator2} formally.  We have not been careful
to specify the domain of $\mc{L}$, but it is clear from this argument
that this derivation is valid for all $h\in C^1\cap L^\infty$. As is
shown in~\cite{Davis-1993}, we can then extend the domain of the
generator using~\eqref{eq:martingale} to encompass all polynomials and
indicator functions.  Strictly speaking, this means that the new
generator is an extension of the previous one, and we will use this
without further comment in the sequel.

\subsection{Moment equations}\label{sec:moments}

A quick perusal of~\eqref{eq:generator2} makes it clear that if we
assume that $f(q,\cdot)$, $\lambda_{q'}(q,\cdot)$,
$\psi_{q'}(q,\cdot)$ are polynomials in $x$, then the right-hand side
sends polynomials to polynomials.  More precisely, if $h(q,x)$ is any
function that is polynomial in $x$, then $\mc{L}h(q,x)$ is also
polynomial in $x$.  In particular, we denote $h^{(m)}_{q'}(q,x) =
x^m\delta_{q,q'}$, and we see that $\mc{L}h^{(m)}_q$ is a polynomial
in $x$, so that Dynkin's formula becomes an (infinite-dimensional)
linear ODE on the set $\{h^{(m)}_q\}_{m\in\N^d,q\in\mc{Q}}$. However,
the first challenge that we obtain is clear: if the degree of any
$f(q,\cdot)$ is greater than one, or the degree of any of the
$\lambda_{q'}(q,\cdot)$ or $\psi_{q'}(q,\cdot)$ are positive, then we
see that the degrees of the terms on the right-hand side of the
equation are higher than those on the left, and we are thus led to the
problem of {\em moment closure}. On the other hand, we still have a
linear system, even if it is infinite-dimensional.  Thus, we might
hope to make sense of this flow by writing down a semigroup on a
reasonable function space.  In fact, we show in
Section~\ref{sec:paradox} that this is in general not possible.

\subsection{Radial components}\label{sec:radial}

Consider the case where $\mathcal{P}_q = \R^d$ for all $q$, and we can
abuse notation slightly and identify $\R^d$ as the continuous state
space.  Consider the ODE 
\begin{equation*}
  \frac{d}{dt} X_t = f(Q_t,X_t)
\end{equation*}
for $X_t\in\R^d$.  Write $X_t = R_t U_t$, for $R_t\in[0,\infty)$ and
$U_t\in S^{d-1}$.  Since $R_t^2 = \sum_{k=1}^d X_{t,k}^2$, we see that
\begin{equation*}
  2R_t \frac{d}{dt}R_t = \sum_{k=1}^d 2 X_{t,k}\frac{d}{dt}X_{t,k} = \sum_{k=1}^d 2 X_{t,k}f_k(Q_t,X_t),
\end{equation*}
or
\begin{equation*}
  \frac{d}{dt}R_t = \frac{\sum_{k=1}^d 2 X_{t,k}f_k(Q_t,X_t)}{\sum_{k=1}^d X_{t,k}^2} =: \rho(Q_t,X_t).
\end{equation*}
In general, this depends on both $R_t$ and $U_t$ through $X_t$, and we
cannot project the dynamics onto $R_t$.  However, if we assume that
$\rho(q,x)$, $\lambda_{q'}(q,x)$, and $\psi_{q'}(q,x)$ are independent
of the angular component of $x$, then we can consider a purely radial
model for the dynamics and consider an SHS whose continuous state space
is the positive real axis.

More generally, consider the following: consider the system $\SHS(\mc
Q, \mc \R^d, \Lambda, \Psi, f)$, where we assume $\psi_{q'}(q,x) =
\gamma_{q,q'}x$, and define
\begin{align*}
  \overline\rho(q,r) = \sup_{\av x = r}\rho(q,x), \quad   \underline\rho(q,r) = \inf_{\av x = r}\rho(q,x),\\
  \overline\lambda_{q'}(q,r) = \sup_{\av x = r}\lambda_{q'}(q,x), \quad   \underline\lambda_{q'}(q,r) = \inf_{\av x = r}\lambda_{q'}(q,x),
\end{align*}
and further assume that the system $\SHS(\mc Q, \mc \R^d, \Lambda,
\Psi, f)$ is ergodic on sets of constant radius, then we can use the
monotonicity results of Lemma~\ref{lem:monotone} to show, for example,
that if the one-dimensional system $\SHS(\mc Q, \R^+,
\overline\Lambda, \Psi, \underline\rho)$ has finite-time blowups
w.p.~1, then so does the system $\SHS(\mc Q, \mc \R^d, \Lambda, \Psi,
f)$, and similarly for the converse statement: any moments of
$\SHS(\mc Q, \R^+, \underline\Lambda, \Psi, \overline\rho)$ that are
finite are also finite for $\SHS(\mc Q, \mc \R^d, \Lambda, \Psi, f)$.
The remainder of this paper can be thought of as concentrating on the
case of radial flows described above, and we will assume throughout
the remainder that our state space is $\R^+$.

\section{Moment Closure and Convexity  ---  One state}\label{sec:mc1}

One of the main results of this paper is that under certain conditions
on the growth of the functions $f, \lambda$ as $x\to\infty$, the
moment equations are well-behaved and useful.  We will first present
the simpler case where there is one state, i.e., $\av{\mc Q} = 1$; 
there is one reset map $\psi$; and the state space $\mc P$ is
one-dimensional, and is, in fact, the positive reals.  Due to this
simplification, we now say that $X_t = \SHS( \lambda, \psi,f)$, where
again we have the flow map $\varphi$:
\begin{equation*}
  \frac{d}{dt}\varphi^t(x) = f(\varphi^t(x)),\quad \varphi^0(x) = x.
\end{equation*}
We want to consider the case where the ODE is unstable at the origin,
so $f(0) = 0$, $f'(0) > 0$, and to balance this we want the reset map
to move towards the origin, so that $\psi(x) \le x$.  Since $\psi$
will be a polynomial, we need to choose $\psi(x) = \gamma x$ with
$\gamma \in [0,1]$.  We also assume that $f(x), \lambda(x)$ are
polynomials of degrees $\df, \dl$ respectively.

In words, we are assuming that the rate is positive for positive $x$,
and that the ODE given by $f$ has a repelling fixed point at the
origin.  We allow for $f$ to be superlinear, so that it could lead to
finite-time blowup on its own.  One natural objective is to determine
what properties of $\lambda$ will ensure no finite-time blowups.

\subsection{Moment flow equations}

The generator for $X_t = \SHS(\lambda, \gamma x,f)$ is 
\begin{equation*}
  \mc L h(x) = f(x) \frac{dh}{dx}(x) + \lambda(x) (h(\gamma x) - h(x)).
\end{equation*}
The test functions that we are interested in are $h^{(m)}(x) := x^m$,
and we want to study the motion of the moments $\mu_m := \E[X_t^m] =
\E[h^{(m)}(X_t)]$.  Plugging this in, we obtain
\begin{equation*}
  \mc L h^{(m)}(x) = f(x)\frac m x h^{(m)}(x) + \lambda(x) (\gamma^m-1)h^{(m)}(x).
\end{equation*}
Since $f(0) = 0$, $f(x)/x$ is a linear combination of terms of
nonnegative degree, so the first term is a polynomial with all powers
at least $m$. If we write
\begin{equation*}
  f(x)  = \sum_{\ell=1}^{\df}\alpha_\ell x^\ell,\quad  \lambda(x)  = \sum_{\ell=1}^{\dl} \beta_\ell x^\ell,
\end{equation*}
then, by taking expectations, we have the {\bf moment flow}
equations:
\begin{equation}\label{eq:momentflow}
  \frac{d}{dt}\mu_m = \sum_{\ell = m}^{m+\dl} C_{m,\ell} \mu_\ell,
\end{equation}
where
\begin{equation}\label{eq:Cml}
\begin{split}
  C_{m,m} &= m\alpha_1,\quad C_{m,m+\dl} = \beta_{\dl}(\gamma^m-1),\\
  C_{m,\ell} &= m \alpha_{\ell+1} + \beta_\ell(\gamma^m-1).  
\end{split}
\end{equation}

It is not hard to see that $C_{m,m+\dl}<0$ for all $m>0$, and
$C_{m,\ell} > 0$ for $m \le \ell < m+\dl-1$ and $m$ sufficiently large
(in fact, $C_{m,\ell}\to\infty$ linearly in $m$ for any fixed $\ell$).
Recall that since we assume $f'(0) > 0$, this means that $\alpha_1>0$.
Under these assumptions, we can state the theorem:

\begin{theorem}\label{thm:mc1}
  If $\df \le \dl$, then the moment flow equations have a globally
  attracting fixed point.  The set of fixed points
  for~\eqref{eq:momentflow} form a $\dl$-dimensional linear manifold
  parameterized by an infinite family of linear equations, and thus we
  cannot find this fixed point simply by using the algebraic relations
  of~\eqref{eq:momentflow}.
\end{theorem}

We delay the proof of the theorem until later, but first we present a
paradox that makes the result a bit surprising.

\subsection{A Paradox of Posedness}\label{sec:paradox}

Using the results of~\cite{Meyn.Tweedie.book}, it follows that there
exists a unique invariant measure to which the system converges at an
exponential rate.  Choose $h(x) = x$, then $\mc L h(x)$ is a
polynomial whose leading coefficient is negative.  Thus, there is a
$b$ with $\mc L h \le -h + b \mathbf{1}_{\mc{S}}$, where $\mc{S}$ is a
compact subset of the positive reals, and thus we have a unique
invariant measure to which we converge exponentially
quickly~\cite[Theorem 14.0.1]{Meyn.Tweedie.book}.

This suggests that the moment flow equations~\eqref{eq:momentflow}
should be well-behaved and tend to an equilibrium solution in some
limit.  However the moment flow equations {\em sui generis} are
ill-posed in a very precise sense:

\begin{definition}[Ill-posed]\label{def:ill}
  Given a state space $X$ and a flow map $\varphi\colon X\times \R \to
  X$, we say that the flow is ill-posed if, for any $x_0\in X$, any
  $t>0$, and any $\e>0$, there is a $y\in X$ with $\av{x_0-y} < \e$
  and $\av{\varphi(t,x_0)-\varphi(t,y)} > 1$.  We will slightly abuse
  notation and say an ODE is ill-posed if its flow map is ill-posed.
\end{definition}

Then we have the following:

\begin{proposition}[Ill-posedness]\label{prop:ill}
  The linear system~\eqref{eq:momentflow} does not generate a strongly
  continuous semigroup on any $\ell^p$ space (or, in fact, on any
  subspace of $\R^\infty$ that contains vectors of finite support).
  Specifically, the ODE is ill-posed in the sense of
  Definition~\ref{def:ill} on any of these spaces.
\end{proposition}

\begin{proof}
  Noting that the system is upper-triangular, and that the diagonal
  elements increase without bound, it is not hard to see that the
  spectrum of this matrix should be unbounded.  To be more specific:
  if we consider any $M\times M$ truncation of this matrix, it has
  eigenvalues $C_{m,m}$ with $m=1,\dots M$. It also has a basis of
  eigenvectors, which we will call $v_1,\dots, v_M$.  These
  eigenvectors embed into $\R^\infty$ in the obvious way and are thus
  eigenvalues of the matrix $A$. Let $V$ be any linear space
  containing all of the $v_k$ (NB: any $\ell^p$ space, with $1\le p
  \le \infty$ would be appropriate).  Then with $A$ considered as a
  linear operator from $V$ to itself, $C_{m,m}\in\Spec(A)$.  Recall
  from above that $C_{m,m}\to\infty$ linearly fast in $m$.  By the
  Hille--Yosida Theorem~\cite[\S 142,143]{Riesz.Sznagy.book}, this
  flow posed on $V$ does not generate a strongly continuous semigroup.
\end{proof}

\begin{remark}
  In particular, we have shown that every finite-dimensional
  truncation of the moment flow equations is linearly unstable, and in
  fact the instability worsens with the order of the truncation.
\end{remark}

\subsection{Convexity to the rescue}

The above certainly seems paradoxical.  One method from stochastic
processes tells us that the behavior of a system is well-behaved as
$t\to\infty$ (in fact, has a globally attracting fixed point), yet, on
the other hand, a method from linear analysis tells us the system is
quite ill-behaved (being ill-posed on finite time domains).

However, the linear system~\eqref{eq:momentflow} does not contain all
of the information given by the problem.  This flow is given by the
flow of moments of a stochastic process.  In this light, the vectors
$v_k$ from the proof above are ``illegal'' perturbations, because
there can be no random variable, supported on the positive reals,
whose moments are given any $v_k$ as defined in the proof of
Proposition~\ref{prop:ill}. In fact, there can be no random variable
whose moments are given by a vector with entries that go to zero, or
even has bounded entries.  This is due to Jensen's
Inequality~\cite{Billingsley.book}:

\begin{theorem}[Jensen's Inequality]
 If $g(\cdot)$ is any convex function and $X$ a real-valued random
 variable, then
\begin{equation}\label{eq:Jensen}
  g\left(\E[X]\right)\le \E[g(X)],
\end{equation}
with equality only if $X$ is ``deterministic'', i.e., the distribution
for $X$ is an atom.  In particular, since $g(x) = x^p$ is convex on
$[0,\infty)$ for any $p>0$, if $X\ge0$, then
\begin{equation*}
  \mu_q^{p/q} = \left(\E[X^q]\right)^{p/q} \le \E[X^p] = \mu_p,
\end{equation*}
or, equivalently,
\begin{equation}\label{eq:momentpq}
  \mu_p \ge \mu_q^{p/q}\mbox{ for all } p\ge q.
\end{equation}
\end{theorem}

In particular, this means that the moments are not actually
independent coordinates in some vector space, and we are not allowed
arbitrary perturbations of a fixed point.  For example,
reconsidering~(\ref{eq:momentflow},~\ref{eq:Cml}), we see that the
largest moment has a negative coefficient, and from Jensen this grows
superlinearly with respect to $\mu_m$.  This will be enough to prove
stability:

\begin{lemma}\label{lem:bounded1}
  If we assume the equations~\eqref{eq:momentflow}, and, further, that
  $\mu_m$ are the moments of a random variable, then as long as
  \begin{equation}\label{eq:ineq}
    \mu_m > \left(\frac{\dl C_{m,\ell}}{|C_{m,m+\dl}|}\right)^{m/(m-\ell + \dl)}
  \end{equation}
  for each $C_{m,\ell}$ that is positive, then the right-hand side
  of~\eqref{eq:momentflow} is negative, and thus $\mu_m$ is
  decreasing.  Since $C_{m,m+\dl}<0$, we have that $\ell \in
  \{m,\dots, m+\dl-1\}$, so the denominator in the exponent
  in~\eqref{eq:ineq} is always positive. Since this puts a finite
  number of constraints on $\mu_m$, this means that $\mu_m$ is
  inflowing on a compact subset of $\R^+$ and thus has bounded
  trajectories.
\end{lemma}

\begin{proof}
  First recall that $C_{m,m+\dl} < 0$.  Choose an $\ell$ such that
  $C_{m,\ell} > 0$ (if none such exist, then we are done.)  If
 \begin{equation*}
    \mu_m > \left(\frac{\dl C_{m,\ell}}{|C_{m,m+\dl}|}\right)^{m/(m-\ell + \dl)},
  \end{equation*}
then
\begin{equation*}
  \mu_m^{\frac{m-\ell+\dl}m} > \frac{\dl C_{m,\ell}}{|C_{m,m+\dl}|}.
\end{equation*}
Since 
\begin{equation*}
  \mu_\ell \ge \mu_m ^{\ell/m},
\end{equation*}
this means 
\begin{align*}
  C_{m,\ell}  - \frac{|C_{m,m+\dl}|}{\dl}\mu_\ell^{(m+\dl-\ell)/\ell} &< 0\\
  C_{m,\ell}\mu_\ell  - \frac{|C_{m,m+\dl}|}{\dl}\mu_\ell^{(m+\dl)/\ell} &< 0.
\end{align*}
By Jensen again, this implies that
\begin{equation*}
  C_{m,\ell}\mu_\ell  - \frac{|C_{m,m+\dl}|}{\dl}\mu_{m+\dl}< 0.
\end{equation*}
(If $C_{m,\ell} < 0$, then the previous inequality is satisfied
trivially.)  From this we have
\begin{equation*}
  \sum_{\ell = m}^{m+\dl -1}C_{m,\ell}\mu_\ell< \sum_{\ell = m}^{m+\dl -1}\frac{|C_{m,m+\dl}|}{\dl}\mu_{m+\dl} = |C_{m,m+\dl}|\mu_{m+\dl},
\end{equation*}
and so
\begin{equation*}
  \sum_{\ell = m}^{m+\dl -1}C_{m,\ell}\mu_\ell + C_{m,m+\dl}\mu_{m+\dl} < 0.
\end{equation*}
\end{proof}

Now, we are ready to provide a formal proof to the result in Theorem~\ref{thm:mc1}.

{\bf Proof of Theorem~\ref{thm:mc1}.}

By Lemma~\ref{lem:bounded1}, the solution of~\eqref{eq:momentflow} has
bounded trajectories. In particular, choosing
\begin{equation*}
  M_m = \max_{m \le \ell \le m+\dl -1} \left(\frac{\dl C_{m,\ell}}{|C_{m,m+{\dl}}|}\right)^{m/(m-\ell+\dl)},
\end{equation*}
we have $\mu_m(t) \le M_m$ for $t$ sufficiently large.

We also note that the steady-state solution of this system satisfies
the (infinite) family of linear equations
\begin{equation*}
  \sum_{\ell = m}^{m+\dl} C_{m,\ell} \mu_\ell=0,
\end{equation*}
or
\begin{equation*}
  \mu_{m+\dl} = -\frac{1}{C_{m,m+\dl}} \sum_{\ell = m}^{m+\dl-1} C_{m,\ell} \mu_\ell.
\end{equation*}
For $m$ sufficiently large, all of the coefficients on the right-hand
side of the equation are positive.

Then. this linear system has exactly $\dl$ degrees of freedom.  Choose
$m^\star$ so that $C_{m,\ell} >0$ for $\ell = m,\dots, m+\dl-1$.
Then, if $\mu_{m^\star},\dots, \mu_{m^\star+\dl}$ are given, then the
$1$st equation gives a unique solution for $\mu_{m^\star+\dl+1}$, and
then the second equation would give a unique solution for
$\mu_{m^\star+\dl+2}$, and so on.
\qed

So, to summarize: we write down the moment flow equations from
Dynkin's formula, giving us an infinite-dimensional linear system.  We
then note that if this system is considered as a linear system with no
further structural information, then the linear system is ill-posed in
any reasonable function space, even though the stochastic process
converges to a unique invariant measure.  However, if we add on the
minimal constraint that these variables are the moments of a
stochastic process, then this is exactly what we need to make the
equations well-posed and to converge to a reasonable limit.  Yet
again, on the other hand, these equations are always degenerate in the
sense that there is a $\dl$-dimensional linear manifold of fixed
points.

\subsection{Case study I: one state + linear}

In this section, we go through all of the above computations for the
simplest possible case: the ODE, the reset map, and the jump rate are
all linear.  More specifically, we consider the system where the flow
is $dx/dt = \alpha x$, the jump rate is $\beta x$, and the reset map
is $x\mapsto \gamma x$.  Here we assume that $\alpha, \beta > 0$ and
$\gamma\in[0,1]$ (the two extreme cases of $\gamma\in\{0,1\}$ are easy
to solve explicitly).

The generator for this process is given by
\begin{equation}\label{eq:ABGgen}
  \mc Lh(x) = \alpha x \frac{dh}{dx} + \beta x \Big(h(\gamma x) - h(x)\Big). 
\end{equation}
Again defining the $k$th moment of $X_t$ by $\mu_k(t) := \E[X^k_t] =
\E[h^{(m)}(X_t)]$, we obtain
\begin{equation}\label{eq:ABGODE}
  \dot\mu_m(t) = \alpha m\mu_m + \beta(\gamma^m-1)\mu_{m+1}.
\end{equation}
If we consider this system simply as a linear system, then we run into
the same problems of ill-posedness as before; the diagonal elements of
the matrix are $\alpha, 2\alpha,\dots, N\alpha,\dots$ which go to
infinity along the positive real axis.

As before, we have Jensen coming to the rescue, and we can deduce that
if, for any $k$,
\begin{equation}\label{eq:lbk}
  \mu_k > \left(\frac{\alpha k}{\beta(1-\gamma^k)}\right)^k,
\end{equation}
then $d\mu_k/dt < 0$.  Considering the fixed point
of~\eqref{eq:ABGODE} gives us that, if we define
\begin{equation*}
  \tmu_k = \E[X_\infty^k],
\end{equation*}
then $\tmu_k$ are constant in time, and are solutions
to~\eqref{eq:ABGODE}, so that we have
\begin{equation}\label{eq:recursion}
  \tmu_{k+1} = \frac{\alpha k}{\beta(1-\gamma^k)} \tmu_k,
\end{equation}
and from~\eqref{eq:lbk} we have the bound that
\begin{equation}\label{eq:ubk}
  \tmu_k \le \left(\frac{\alpha k}{\beta(1-\gamma^k)}\right)^k.
\end{equation}
From this, we obtain a recursive relationship for all of the moments;
as long as we know $\tmu_1$, then we know $\tmu_k$ for all $k$.  As it
stands, $\tmu_1$ is free, subject to the bound~\eqref{eq:ubk} of being
less than $\alpha/(\beta(1-\gamma))$.

We might seek to get more bounds, since~\eqref{eq:ubk} comes from
considering only the relation between $\tmu_m$ and $\tmu_{m+1}$.  Let
us define
\begin{equation*}
  c_k = \frac{\alpha k}{\beta(1-\gamma^k)}, \quad  c_{a,b} = \prod_{\ell = a}^{b-1} c_\ell.
\end{equation*}
Iterating~\eqref{eq:recursion} gives $\tmu_{k+j} = c_{k,k+j} \tmu_k$,
and using Jensen's inequality again, this gives
\begin{equation*}
  \tmu_k^{(k+j)/k} \le c_{k,k+j} \tmu_k\mbox{,  \ \ or,\ \ } \tmu_k^{j/k} \le c_{k,k+j}.
\end{equation*}
Since $\tmu_1 \le \tmu_k^{1/k}$, this gives
\begin{equation*}
  \tmu_1 \le c_{k,k+j}^{1/j}.
\end{equation*}
However, it is not hard to see that this last quantity is just the
geometric mean of $c_k,c_{k+1},\dots, c_{k+j-1}$.  Thus we see that
all of this work has given us no useful information: since the $c_k$
are increasing in $k$:
\begin{equation*}
  \frac{d}{dk}c_k(\gamma) = \frac{1+\gamma^k(k\ln\gamma-1)}{(1-\gamma^k)^2},
\end{equation*}
this means that all of these geometric means are worse upper bounds than $c_1$.

It turns out that we can get one more bit of information from the
generator, namely: choose $h(x) = \ln(x)$:
\begin{equation*}
  \mc{L} h(x) = \alpha x \frac 1 x + \beta x (\ln(\gamma x)-\ln(x)) = \alpha + (\beta\ln\gamma) x.
\end{equation*}
If our system is in equilibrium, then $\E[h(X_t)]$ is constant in
time, so we have 
\begin{equation*}
  0 = \alpha + \beta\ln\gamma \E[X_\infty],\quad \E[X_\infty] = \frac{-\alpha}{\beta\ln\gamma} = \frac{\alpha}{\beta\ln(1/\gamma)}.
\end{equation*}

\subsection{Maximum entropy formulation}\label{sec:me}

We know that the invariant distribution needs to satisfy the
infinitely many conditions on its moments that
\begin{equation*}
  \sum_{\ell=m}^{m+\dl} C_{m,\ell} \tmu_\ell = 0,
\end{equation*}
but that, generically, this gives rise to a $\dl$-dimensional linear
manifold of fixed points for the moment flow.  

Another way of writing this is that if we define the polynomials 
\begin{equation*}
  f_m(x) = \sum_{\ell=m}^{m+\dl} C_{m,\ell} x^\ell,
\end{equation*}
then we have
\begin{equation}\label{eq:defofpk}
  \E[f_k(X_\infty)] = 0,~~\forall k.
\end{equation}
This leads to a natural conjecture: if we define $Z$ as the random
variable that has the maximum entropy subject to satisfying
$\E[f_m(Z)] = 0$ for all $m$, then the law of $Z$ should be the same
as that of $X_\infty$.  More precisely, if we define the entropy of a
probability distribution by
\begin{equation*}
  H(\pi) = -\int_0^\infty \pi(x) \ln \pi(x) \,dx,
\end{equation*}
 define $\mc P_*$ as the set of all probability distributions that
 satisfy~\eqref{eq:defofpk}, and define
\begin{equation*}
  \pi_* = \sup_{\pi\in\mc P_*} H(\pi),
\end{equation*}
then the probability distribution of $X_\infty$ should be $\pi_*$.
In~\cite{CDC-ZDDD.14}, we present a method for efficient computation
of the maximum entropy distribution, and show that the maximum entropy
distribution is the invariant distribution for a wide variety of test
cases.

\subsection{Numerical results}

\newcommand{\dtmax}{\Delta t_{\mathsf{max}}}

In Figure~\ref{fig:One}, we plot the results of a Monte Carlo
simulation of a one-state system as described above, where we choose
\begin{equation*}
  f(x) = x^2,\quad \lambda(x) = 2x^2,\quad \gamma = 1/2.
\end{equation*}
We simulated $10^3$ realizations of the process, and ran each one
until $10^4$ numerical steps occurred.  The method we used was a
hybrid Gillespie--1st order Euler method: we choose and fix a
$\dtmax\ll 1$.  Given $X_t$, we have $\lambda(X_t)$, and we determine
the time of the next jump as $dt = -\log(U)/\lambda(X_t)$, where $U$
is a uniform $[0,1]$ random variable.  If $dt < \dtmax$, then we
integrate the ODE using the 1st-order Euler method for time $dt$
(i.e., we set $X_{t+dt} = X_t + dt \cdot f(X_t)$), then multiply by
$\gamma$.  If $dt \ge \dtmax$, then we integrate the ODE for time
$\dtmax$ and do not jump.  It is clear that in the limit as
$\dtmax\to0$, this converges in every sense to the stochastic process,
and it is also not hard to see that this is equivalent to computing
the trajectory and next jump using formulas~\eqref{eq:defofT}
and~\eqref{eq:path} by discretizing the integral in~\eqref{eq:defofT}
using timesteps of $\dtmax$.

In Figures~\ref{fig:One}a and ~\ref{fig:One}b we plot one realization of the process, in
linear and in log coordinates.  To the eye, $\log(X_t)$ looks almost
like an Ornstein--Uhlenbeck process, and this is borne out by the
distribution in Figure~\ref{fig:One}, where we see that the invariant
distribution of $X_t$ is very close to log-normal, at least to the
eye.  To check this, we plot a QQ-plot of $\log(X_t)$ versus a normal
distribution with the same mean and variance, and we see that the
distribution is not quite log-normal  ---  the fact that this plot is
concave down means that this distribution is a little ``tighter'' than
a normal distribution.  Thus the distribution of $\log(X_t)$ looks
like a Gaussian up to two or three standard deviations, but has
smaller tails.  In any case, to verify that this is not just due to
sampling error, one can plug the general log-normal into the formal
adjoint $\mc{L}^*$ derived from~\eqref{eq:defofL}, and see by hand
that log-normals are not in the nullspace.

\begin{figure}[ht]
\begin{centering}
  \includegraphics[width=0.97\textwidth]{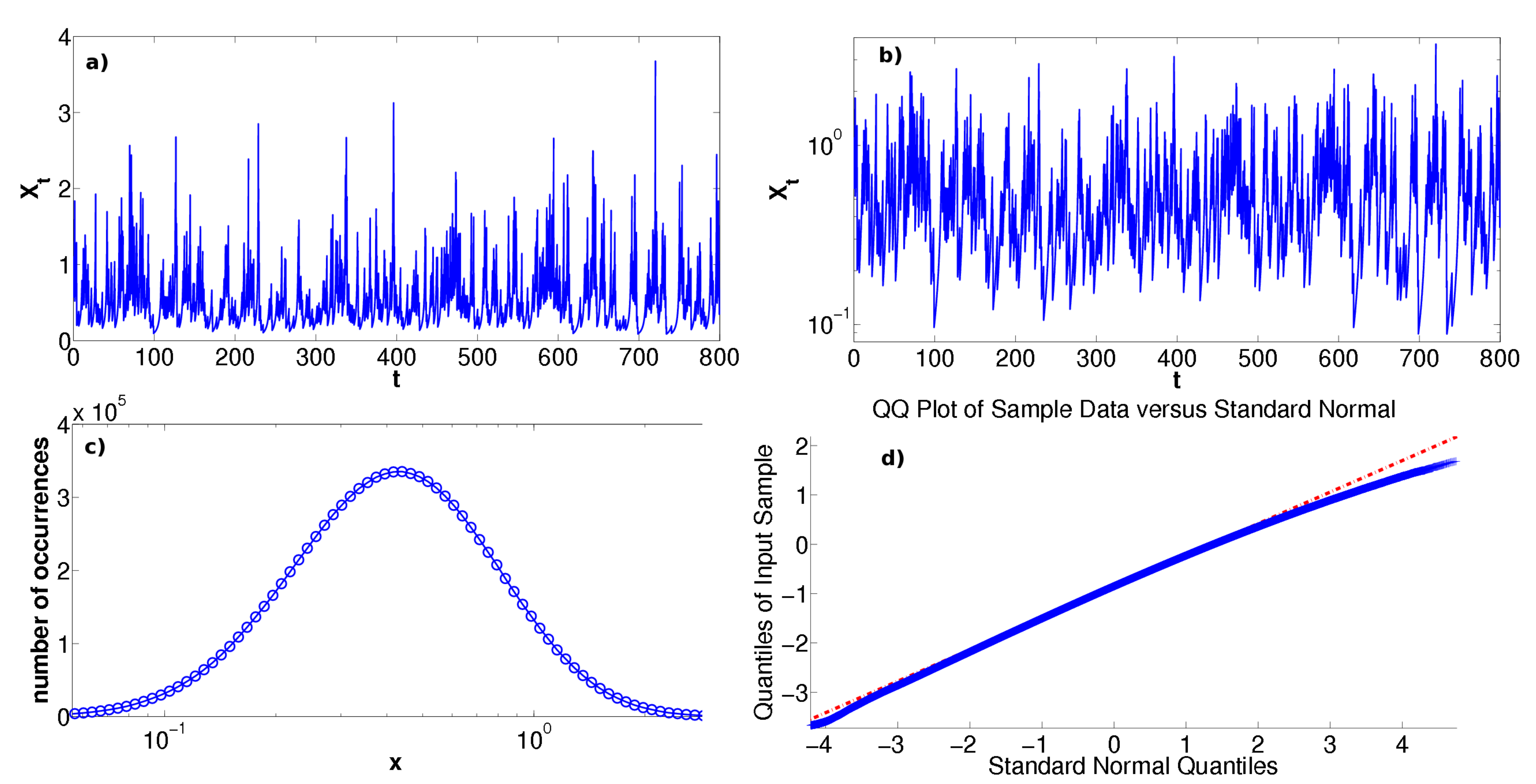}
  \caption{Plots of realizations and histograms of one-state system,
    see description in text for more detail.}
  \label{fig:One}
\end{centering}
\end{figure}

\section{Moment closure and convexity  ---  Multiple states}\label{sec:mc2}

We now extend the results of the previous section to the case where
there are multiple states, i.e., $\av{\mc Q} > 1$.  The argument
behind the method used here is the same as before: when we write down
the moment equations, if we can show that if the right-hand side of
the evolution equation for each moment has a term of higher degree,
and this term has a negative coefficient, then the previous approach
works as well.  As before, we will assume that our state space is the
positive reals, and that our reset maps are linear, i.e.
$\psi_{kl}(x) = \gamma_{kl}x$.

The main results of this section are Theorems~\ref{thm:bounded}
and~\ref{thm:unbounded}; the former gives sufficient conditions for
the moment flow system to have bounded orbits, and the latter gives
sufficient conditions for the moment flow system to have unbounded
orbits.  The technique used in the proofs of these theorems is, in principle, the
same as in the previous section: there we had a case where, in each
equation, there was a term involving the higher-order moment with a
negative coefficient.  In this case, we will show that there is a term
on the right-hand side that plays the same role as this negative
coefficient, but in this case, since the $m$th moment is effectively a
vector of length $\av{\mc Q}$, this will be a $\av{\mc Q}\times
\av{\mc Q}$ matrix whose sign-definiteness will establish stability
(or the lack thereof).

\subsection{ Moment flow equations}

\renewcommand{\deg}{{\mathsf{deg}}}
\newcommand{\lc}{{\mathsf{LC}}}
 We recall 
 \begin{equation}\tag{\ref{eq:generator2}}
   \mc{L}h(q,x) = f(q,x)\cdot \nabla_x h(q,x) + \sum_{q'\in\mc{Q}} \lambda_{q'}(q,x)\left(h(q',\psi_{q'}(q,x)) - h(q,x)\right).
 \end{equation}
Let us define, for $\theta\in\mc Q$ and $m\in \N$,
\begin{equation*}
  h_\theta^{(m)}(q,x) := 1_\theta(q)x^m = \delta_{q,\theta}x^m,\quad \mu_\theta^{(m)} := \E[h_\theta^{(m)}(Q_t,X_t)].
\end{equation*}
Note that
\begin{equation*}
  \sum_{\theta\in \mc Q} h_\theta^{(m)} = x^m,\mbox{ and thus } \mu^{(m)} := \sum_{\theta\in \mc Q} \mu_\theta^{(m)} = \E[X_t^m]
\end{equation*}
is the total $m$th moment of $X_t$.  One can think of
$\mu_\theta^{(m)}$ as the conditional $m$th moment of $X_t$,
conditional on $Q_t = \theta$, times the probability that
$Q_t=\theta$.  We plug $h_\theta^{(m)}$ into~\eqref{eq:generator2}, to
obtain
\begin{align*}
  \mc{L} h_\theta^{(m)}(q,x) 
  &= f(q,x)\cdot \nabla_x h_\theta^{(m)}(q,x)  + \sum_{\ell\in \mc Q}\lambda_{\ell}(q,x) h_\theta^{(m)}(\ell,\psi_{\ell}(q,x)) - \sum_{\ell\in \mc Q}\lambda_{\ell}(q,x) h_\theta^{(m)}(q,x)\\
  &= f(q,x) 1_\theta(q) \cdot m \cdot x^{m-1} + \sum_{\ell\in \mc Q}\lambda_{\ell}(q,x)1_\theta(\ell)(\gamma_{q,\ell}x)^m) - \sum_{\ell\in \mc Q}\lambda_{\ell}(q,x)1_\theta(q) x^{m}.
\end{align*}
We plug in $(Q_t,X_t)$ for $(q,x)$ and take expectations for each of
the three pieces separately.  
\begin{equation}\label{eq:expqtxt}
\begin{split}
  &\E[\mc {L} h_\theta^{(m)}(Q_t,X_t)] \\&\qquad= m \E[1_\theta(Q_t) f(Q_t,X_t) X_t^{m-1}] + \E[\sum_{\ell\in\mc Q}\lambda_{\ell}(Q_t,X_t)\gamma_{Q_t,\ell}^m1_\theta(\ell)X_t^m] - \E[\sum_{\ell\in \mc Q}\lambda_{\ell}(Q_t,X_t)1_\theta(Q_t)X_t^m].
\end{split}
\end{equation}
The first and third terms in the right-hand side of \eqref{eq:expqtxt} are more or less straightforward, but the
second term can be simplified in the following manner:
\begin{align*}
  \E\left[\sum_{\ell\in\mc Q}\lambda_{\ell}(Q_t,X_t)\gamma_{Q_t,\ell}^m1_\theta(\ell)X_t^m\right]
  &= \E[\lambda_\theta(Q_t,X_t)\gamma_{Q_t,\theta}^m X_t^m]\\
  &= \sum_{\ell \in \mc{Q}} \E[1_\ell(Q_t) \lambda_\theta(Q_t,X_t)\gamma_{\ell,\theta}^m X_t^m],\\
\end{align*}
and thus we have
\begin{equation}\label{eq:multimf}
\begin{split}
  &\E[\mc {L} h_\theta^{(m)}(Q_t,X_t)]\\
  &\quad= m \E[1_\theta(Q_t) f(Q_t,X_t) X_t^{m-1}] +\sum_{\ell \in \mc{Q}} \E[1_\ell(Q_t) \lambda_\theta(Q_t,X_t)\gamma_{Q_t,\theta}^m X_t^m] - \sum_{\ell\in \mc Q}\E\left[\lambda_{\ell}(Q_t,X_t)1_\theta(Q_t)X_t^m\right].  
\end{split}
\end{equation}

\begin{example}\label{ex:linear}

  The general formula \eqref{eq:multimf} is a bit complicated to parse, therefore we
  first work out a particular test case.  Assume that all of the
  functions involved are linear, i.e.
\begin{equation*}
  f(q,x) = \alpha_q x,\quad \lambda_{\ell}(q,x) = \beta_{q,\ell}x,\quad \psi_{\ell}(q,x) = \gamma_{q,\ell}x, 
\end{equation*}
then we obtain
\begin{equation*}
  \E[\mc {L} h_\theta^{(m)}(Q_t,X_t)]= m \alpha_\theta\E[1_\theta(Q_t) X_t^{m}] + \sum_{\ell\in \mc{Q}} \beta_{\ell,\theta}\gamma_{\ell,\theta}^m\E\left[1_\ell(Q_t)X_t^{m+1}\right] - \sum_{\ell\in \mc Q}\beta_{\theta,\ell}\E\left[1_\theta(Q_t)X_t^{m\textcolor{blue}{+1}}\right],
\end{equation*}

or
\begin{equation}\label{eq:mflin}
  \frac{d}{dt}\mu_\theta^{(m)}(t) = m\alpha_\theta \mu_\theta^{(m)}(t) + \sum_{\ell\in\mc Q}\beta_{\ell,\theta}\gamma_{\ell,\theta}^m \mu_{\ell}^{(m+1)}(t)- \sum_{\ell\in \mc Q}\beta_{\theta,\ell} \mu_\theta^{(m+1)}(t).
\end{equation}

\end{example}

Compare~\eqref{eq:mflin} to~\eqref{eq:ABGODE} and notice that it has
much the same form: the function we differentiate appears first with a
positive coefficient, also we have two terms of one higher degree
with alternating signs, and the positive term has a $\gamma$ in it.
What is different, and what makes this more complicated, is that the
positive term of higher degree depends on the moments in different
discrete states.  So, while we will be able to use a Jensen-like
argument to get boundedness, we have to be more careful, since we do
not know that there is any relationship between $\mu_q^{(m)}$ and
$\mu_{q'}^{(m+1)}$ from just convexity  ---  and thus we need to
consider the entire vector $\{\mu_q^{(m)}\}_{q\in\mc Q}$.

\begin{definition}
  Let $X_t = \SHS(\mc Q, \mc P, \Lambda, \Psi, F)$.  Let
  $\deg(\Lambda) = \max_{kl}\deg(\lambda_{kl})$, and define
  $\beta_{kl}$ as the coefficient of the term of degree
  $\deg(\Lambda)$ in $\lambda_{kl}(x)$, with the convention that
  $\beta_{kl} = 0$ if $\deg(\lambda_{kl}) < \deg(\Lambda)$.  Then the
  {\bf top matrix of degree $m$} of the system is the matrix $M^{(m)}$
  with coefficients
\begin{equation*}
 M^{(m)}_{kl} = \begin{cases} \gamma^m_{lk}\beta_{lk}, & k\neq l,\\ -\sum_{l\in\mc Q} \beta_{kl},& k = l.\end{cases}
\end{equation*}
\end{definition}

\subsection{Theorems for stability}

\begin{theorem}[Bounded moments]\label{thm:bounded}
  Assume that  $\mc P = \R^+$, and let $X_t = \SHS(\mc Q, \mc P, \Lambda, \Psi,
  F)$.   If $\deg(F) \le \deg(\Lambda)$ and all of the eigenvalues of
  $M^{(m)}$ are negative, then the orbit of the total $m$th moment
  $\mu^{(m)}$ is bounded under the moment flow equations; in
  particular, if $M^{(m)}$ has negative spectrum for all $m\ge 1$,
  then all of the moments have bounded orbits.
\end{theorem}

\begin{proof}
  Let us first consider the case where all of the functions are
  linear, as in Example~\ref{ex:linear}.  Writing the vector
  $\mu^{(m)} = \{\mu_q^{(m)}\}_q$, we can write~\eqref{eq:mflin} as
  \begin{equation}\label{eq:mfvec}
    \frac{d}{dt}\mu^{(m)} = m A\mu^{(m)} + M^{(m)} \mu^{(m+1)},
  \end{equation}
  where $A$ is the diagonal matrix with $A_{qq} = \alpha_q$.  Note
  that every entry in $A$ is positive, so the flow is linearly
  unstable at the origin.

  However, also note that if $\mu^{(m)} \gg 1$, then $\mu^{(m+1)} \gg
  \mu^{(m)}$ by Jensen, which means that~\eqref{eq:mfvec} is dominated
  by the second term.  More precisely, if we assume that $\mu_q^{(m)}
  > 1/\epsilon$ for all $q$, then $\mu_q^{(m+1)} >
  1/\epsilon^{1/m}\mu_q^{(m)}$, and thus we can write
  \begin{equation}\label{eq:eps}
 \e^{1/m}\frac{d}{dt} \mu^{(m)} \le \e^{1/m} A \mu^{(m)} + M^{(m)}\mu^{(m)}.
  \end{equation}

For $\epsilon$ sufficiently small, the first term is dominated.  Since
$M^{(m)}$ has negative spectrum, the flow $\dot z = M^{(m)}z$ is such that 
all $z(t)$'s in the positive octant will asymptotically approach the origin, and, moreover,
this is structurally stable to a sufficiently small perturbation by
the Hartman--Grobman Theorem.  Thus, for $\epsilon$ small enough, all
orbits are attracted to the origin, which means that~\eqref{eq:mfvec}
is inflowing on any ball of sufficiently large radius.
Thus~\eqref{eq:mfvec} has bounded orbits.

  Now we consider the general $f, \lambda$ (recall that we assume
  throughout that $\psi_k(q,x) = \gamma_{qk}x$.  Let us write
  \begin{align*}
    f(q,x) &= \sum_{a=1}^{A_q} \alpha_{a,q} x^a,\\
    \lambda_k(q,x) &= \sum_{b=1}^{B_{q,k}} \beta_{b,qk} x^b.
  \end{align*}
  Then, plugging into~\eqref{eq:multimf}, we have
\begin{align*}
  \E[1_\theta(Q_t) f(Q_t,X_t) X_t^{m-1}]
  &= \E[1_\theta(Q_t)\sum_{a=1}^{A_{Q_t}} \alpha_{a,Q_t} X_t^aX_t^{m-1}]\\
  &= \sum_{a=1}^{A_{\theta}} \alpha_{a,\theta} \E[1_\theta(Q_t)X_t^{a+m-1}] = \sum_{a=1}^{A_{\theta}} \alpha_{a,\theta} \mu_\theta^{(a+m-1)},\\
  \E\left[1_\theta(\ell)\gamma_{\ell,\theta}^m\lambda_{\theta}(\ell,X_t)X_t^m\right] 
  &=\E\left[1_\theta(\ell)\gamma_{\ell,\theta}^m\sum_{b=1}^{B_{\ell,\theta}} \beta_{b,\ell\theta} X_t^bX_t^m\right]=\sum_{b=1}^{B_{\ell\theta}} \gamma_{\ell,\theta}^m\beta_{b,\ell\theta}\mu_\ell^{(b+m)},\\
  \E\left[\lambda_{\ell}(Q_t,X_t)1_\theta(Q_t)X_t^m\right] &=\E\left[\sum_{b=1}^{B_{Q_t,\ell}} \beta_{b,Q_t\ell} X_t^bX_t^m\right]=\sum_{b=1}^{B_{\theta,\ell}} \beta_{b,\theta\ell}\mu_\theta^{(b+m)},
\end{align*}
or
\begin{equation}\label{eq:mfnonlin}
  \frac{d}{dt} \mu_\theta^{(m)} = m\sum_{a=1}^{A_{\theta}} \alpha_{a,\theta} \mu_\theta^{(a+m-1)}+\sum_{b=1}^{B_{\ell,\theta}} \gamma_{\ell,\theta}^m\beta_{b,\ell\theta}\mu_\ell^{(b+m)} - \sum_{b=1}^{B_{\theta,\ell}} \beta_{b,\theta\ell}\mu_\theta^{(b+m)}.
\end{equation}

Recalling the definition of $M^{(m)}$, this means that the last two
terms in \eqref{eq:mfnonlin} can be written as $M^{(m)}
\mu^{(m+\deg(\Lambda))}$  ---  note that the definition of $M^{(m)}$
includes, by design, only those coefficients that are the same degree
as the highest possible degree of all $\lambda_{kl}$.
Thus~\eqref{eq:mfnonlin} can be written as
\begin{equation}\label{eq:mfnlsum}
  \frac{d}{dt}\mu_\theta^{(m)} = \sum_{\ell=m-1}^{m+\deg(\Lambda)-1} A_\ell \mu^{(\ell)} + M^{(m)}
  \mu^{(m+\deg(\Lambda))},
\end{equation}
and clearly the same sort of asymptotic analysis done for the linear
case works here as well, since the first term is strictly dominated in
powers by the last.
\end{proof}

\begin{theorem}[Unbounded moments]\label{thm:unbounded}
  Let $\mc P = \R^+$ and $X_t = \SHS(\mc Q, \mc P, \Lambda, \Psi, F)$.
  \begin{enumerate}
  \item If $\deg(F) \le \deg(\Lambda)$, and
    \begin{enumerate}  
    \item $\ip{M^{(m)}x}y>0$ for all $x,y$ in the positive octant, or
    \item there exist $k,l$ such that $\min(M_{kl}^{(m)},M_{lk}^{(m)})
      > \max(|M_{kk}^{(m)}|, |M_{ll}^{(m)}|)$;
    \end{enumerate}
    then the orbit of the total $m$th moment $\mu^{(m)}$ is unbounded
    under the moment flow equations;
  \item If $\deg(F) > \deg(\Lambda)$, then for $m$ sufficiently large,
    all $m$th moments have unbounded orbits under the moment flow
    equations.
  \item If $\deg(F) > \deg(\Lambda)+1$, then all $m$th moments have
    unbounded orbits under the moment flow equations.
  \end{enumerate}
\end{theorem}

\begin{proof}
  Again consider~\eqref{eq:mfnlsum}.  First assume that
  $\ip{M^{(m)}x}y>0$ for all $x,y$ in the positive octant. Since all
  of the $A_\ell$ in that formula are diagonal with positive entries,
  it follows directly that the vector field is outflowing on every
  circle.

  On the other hand, assume that there exists a $k,l$ such that
  $\min(M_{kl}^{(m)},M_{lk}^{(m)}) > \max(M_{kk}^{(m)},
  M_{ll}^{(m)})$.  Without loss of generality by renumbering, assume
  that we have $\min(M_{12}^{(m)},M_{21}^{(m)}) > \max(|M_{11}^{(m)}|,
  |M_{22}^{(m)}|)$.  Let us again consider the linear case, as the
  nonlinear case is the same. We have
  \begin{align*}
    \frac{d}{dt} \mu_1^{(m)} &= m \alpha_1  \mu_1^{(m)} + M_{12}^{(m)} \mu_2^{(m+1)} - |M_{11}^{(m)}|\mu_1^{(m+1)},\\
    \frac{d}{dt} \mu_2^{(m)} &= m \alpha_2  \mu_2^{(m)} + M_{21}^{(m)} \mu_1^{(m+1)} - |M_{22}^{(m)}|\mu_2^{(m+1)}.
  \end{align*}
  Writing $\nu^{(m)} = e^{-m A t}\mu^{(m)}$, where $A =
  \diag(\alpha_1,\alpha_2)$, we have
  \begin{align*}
    \frac{d}{dt} \nu_1^{(m)} &=  M_{12}^{(m)} \mu_2^{(m+1)} - |M_{11}^{(m)}|\mu_1^{(m+1)},\\
    \frac{d}{dt} \nu_2^{(m)} &= M_{21}^{(m)} \mu_1^{(m+1)} - |M_{22}^{(m)}|\mu_2^{(m+1)}.
  \end{align*}
  and thus 
\begin{equation*}
  \frac{d}{dt} (\nu_1^{(m)} + \nu_2^{(m)}) = (M_{21}^{(m)}- |M_{11}^{(m)}|)\mu_1^{(m+1)} + (M_{12}^{(m)}- |M_{22}^{(m)}|)\mu_2^{(m+1)} > 0,
\end{equation*}
so that this sum is always growing, and thus the corresponding sum for
$\mu^{(m)}$ grows at least exponentially.

Now consider the case where $\deg(F) = \deg(\Lambda)+1$  ---  this
implies that the first term in~\eqref{eq:mfnonlin} is of the same
degree, than all of the other terms.  Thus~\eqref{eq:mfnonlin} looks
like
\begin{equation*}
  \frac{d}{dt}\mu_m^{(m)} =  A \mu^{(m+\deg(\Lambda))} + M^{(m)}\mu^{(m+\deg(\Lambda))} + O(\mu^{(m+\deg(\Lambda)-1)},
\end{equation*}
and since the diagonals of $A$ grow linearly in $m$, for sufficiently
large $M$, all of the coefficients of $A+M^{(m)}$ are positive, and
thus the flow has unbounded orbits.

Finally, if $\deg(F) > \deg(\Lambda) + 1$, then~\eqref{eq:mfnonlin} becomes
\begin{equation*}
  \frac{d}{dt}\mu_m^{(m)} =  A \mu^{(m+\deg(F)-1)} + M^{(m)}\mu^{(m+\deg(\Lambda))} + O(\mu^{(m+\deg(F)-2)},
\end{equation*}
where the dominant term is the $A$ term, which is diagonal with
positive diagonals for all $m$.
\end{proof}

\subsection{Examples and corollaries}

We have shown above that as long as $M^{(m)}$ has all negative
eigenvalues, the $m$th moment is stable.  First we show:

\begin{corollary}
  If $0 < \gamma_{kl} < 1$ for all $k,l$, then $\E[X_t^m]$ is bounded
  above for all $t$ and for all $m$.
\end{corollary}

\begin{proof}
  From Theorem~\ref{thm:bounded}, all we need to show is that $M^{(m)}$
  has all negative eigenvalues.  Recall that
  \begin{equation*}
    M^{(m)}_{kl} = \begin{cases} \gamma^m_{lk}\beta_{lk}, & k\neq l,\\ -\sum_{l\in\mc Q} \beta_{kl},& k = l.\end{cases}
  \end{equation*}
  By the Gershgorin Circle Theorem, the eigenvalues of
  $M^{(m)}$ are contained in the union of the $\av{\mc Q}$ balls
\begin{equation*}
  B_k:= B\left(M^{(m)}_{kk}, \sum_{l=1}^{\av{\mc Q}}\av{M^{(m)}_{kl}}\right),
\end{equation*}
i.e., the $k$th ball is centered at the $k$th diagonal coefficient, and
whose radius is given by the sum of the absolute values of the
off-diagonal terms. Since $\gamma_{kl}<1$, and thus $\gamma_{kl}^m <
1$, we have
\begin{equation*}
  \sum_{l=1}^{\av{\mc Q}}\av{\gamma_{kl}^m M^{(m)}_{kl}} < \sum_{l=1}^{\av{\mc Q}}\av{M^{(m)}_{kl}} = \av{M^{(m)}_{kk}}.
\end{equation*}
\end{proof}

\begin{example}\label{ex:2x2}
  Let us assume that $\av{\mc Q} = 2$, so that 
\begin{equation*}
  M^{(1)} = \l(\begin{array}{cc} -\beta_{12}&\beta_{21}\gamma_{21} \\ \beta_{12}\gamma_{12}& -\beta_{21}\\\end{array}\r).
\end{equation*}
We have $\tr M^{(1)} = -\beta_{12} - \beta_{21} < 0$, and $\det
M^{(1)} = \beta_{12}\beta_{21}(1-\gamma_{12}\gamma_{21})$.  It is not
hard to see that the resulting linear system is a saddle if
$\gamma_{12}\gamma_{21} > 1$, and a sink if $\gamma_{12}\gamma_{21} <
1$.  By Theorem~\ref{thm:bounded}, if $\gamma_{12}\gamma_{21} < 1$
this means that $\E[X_t]$ is bounded above.  Of course, note that the
stability condition for $M^{(m)}$ is $\gamma_{12}^m\gamma_{21}^m < 1$,
and thus stability of the first moment implies stability of all higher
moments.  As we see below, this is a special case only if $\av{\mc Q}
= 2$.

In fact, one can get at this result from other means: notice that in
the $2\times 2$ case, all jumps $1\to 2$ are immediately followed by a
jump $2\to1$, and the aggregate effect of these jumps is to multiply
by $\gamma_1\gamma_2$.  Thus it is clear that the necessary and
sufficient condition for stability is that this product be less than
one.

\end{example}

One of the interesting observations made in the previous example is
that we do not require that all $\gamma_{kl}$ be less than one.  But
for general $\av{\mc{Q}}>2$, if we choose one or more of the
$\gamma_{kl}>1$, then we will see some unbounded moments, as in the
following example.

\begin{example}\label{ex:3x3} 
 Consider the symmetric case where we choose $\beta_{kl} = 1$ for all
 $k,l$, and choose $\gamma_{12}> 1$, but $\gamma_{13}, \gamma_{23} <
 1$, so that we have
\begin{equation*}
  M^{(m)} = \l(\begin{array}{ccc} -2&\gamma_{12}^m&\gamma_{13}^m \\ \gamma_{12}^m&-2&\gamma_{23}^m \\ \gamma_{13}^m&\gamma_{23}^m&-2\end{array}\r).
\end{equation*}
If we further choose $\gamma_{12} + \gamma_{13} < 2$ and $\gamma_{12}
+ \gamma_{23} < 2$, then again by Gershgorin theorem, it follows that $M^{(1)}$ has all
negative eigenvalues, and by Theorem~\ref{thm:bounded}, the moment flow
is stable at first order.

Since $\gamma_{12} > 1$, then for some $m>1$, $\gamma_{12}^m > 2$, and
by Theorem~\ref{thm:unbounded}, the moment flow is unstable at $m$th
order.  Thus we have a scenario where $\E[X_t]$ is bounded above for
all time, but $\E[X_t^m]$ grows without bound.
\end{example}

We saw in Example~\ref{ex:3x3} that we can construct a process where
the mean is bounded above, but some higher moments grow without bound.
In fact, it should be clear that if there is a pair $(k,l)$ with
$\gamma_{kl} > 1$ and $\gamma_{lk} > 1$, then this will generically
occur: for sufficiently large $m$, the $m$th moment is unstable.  This
can lead to some counter-intuitive effects, as seen by the following
example.

\begin{example}\label{ex:N-1}
  Consider the case where $\av{\mc{Q}} = N$, $\gamma_{12} =
  \gamma_{21} > N-1$, and all of other the $\gamma_{kl}$ are
  arbitrarily small (for the purposes of this argument, set them to
  zero), so that if the system ever enters a state other than $1$ or
  $2$, then $X_t$ is set to zero.  Let all $\beta_{kl} = 1$, so that
  the jump rates are all the same.  Specifically, this means that
  whenever there is a jump, the next discrete state is chosen
  uniformly in the others.  Start with $Q_0=1, X_0=1$.  The
  probability of never leaving $Q_t\in\{1,2\}$ after the first $k$
  jumps is then $(N-1)^{-k}$, but the multiplier from these
  transitions is $\gamma_{12}^k$, so that $\E[X_{T_k}]$ grows
  exponentially, even if $\P(X_{T_k} > 0)$ is shrinking exponentially.

  Similarly, we could see that if we choose $\gamma_{12} < N-1 <
  \gamma_{12}^2$, then $\E[X_t]$ would decay exponentially, but
  $\E[X_t^2]$ would grow exponentially, and similarly for higher
  moments.

  In short, this says that to obtain a moment instability, all one
  needs is two states exchanging back and forth, as long as their
  multipliers are large enough, even if the probability of this
  sequence of switching is small  ---  which is the content of Part 1 of
  Theorem~\ref{thm:unbounded}.
\end{example}

When we have a probability distribution with some moments finite and
others infinite, this is called\footnote{There is some disagreement in
  the literature of the use of the term ``heavy-tailed''  ---  some
  authors use the term to mean a random variable that has some
  polynomial moments infinite, some use it to mean a random variable
  with infinite variance, and yet others use it to mean a distribution
  whose mgf does not converge in the right-half plane.  We use the
  first convention, and note by this convention, the log-normal
  distributions seen above would not be considered heavy-tailed by
  this convention.} a {\bf heavy-tailed} distribution
(see~\cite{Embree.Trefethen.99, Dorogovtsev.Goltsev.Mendes.08} for
examples in dynamical systems).  The canonical example of a
heavy-tailed distribution is one whose tail decays asymptotically as a
power law, and as such, are typically associated with critical
phenomena in statistical physics~\cite{Goldenfeld.book, Stanley.99,
  Sethna.Dahmen.Myers.01, Kinouchi.Copelli.06, Mora.Bialek.11,
  Larremore.etal.11, Friedman.etal.12}.  In fact, if we assume that
$p(x)$ is the distribution of the process in equilibrium, that $p(x)
\sim x^{-\alpha}$ as $x\to\infty$, and $X_\infty$ is a realization of
this distribution, then
\begin{equation*}
  \E[X_\infty^m] = \int_0^\infty x^m p(x)\,dx \approx \int_{x_*}^\infty x^{m-\alpha}\,dx
\end{equation*}
will converge iff $m-\alpha < 1$.  From this we deduce that if
$\alpha\in(m,m-1)$, then $\E[X_\infty^m] < \infty$ but
$E[X_\infty^{m+1}] = \infty$.  We will in fact show numerical evidence
of power law tails in SHS dynamics in the next section, which leads us
to conjecture that the extreme case of Example~\ref{ex:N-1} above is
atypical.

\subsection{Numerical results}

In Figure~\ref{fig:Two}, we plot the results of a Monte Carlo
simulation for a two-state SHS.  We simulate $10^3$ realizations of
this system, and each realization was integrated until $10^4$ jumps
had occurred.

\begin{figure}[ht]
\begin{centering}
  \includegraphics[width=0.97\textwidth]{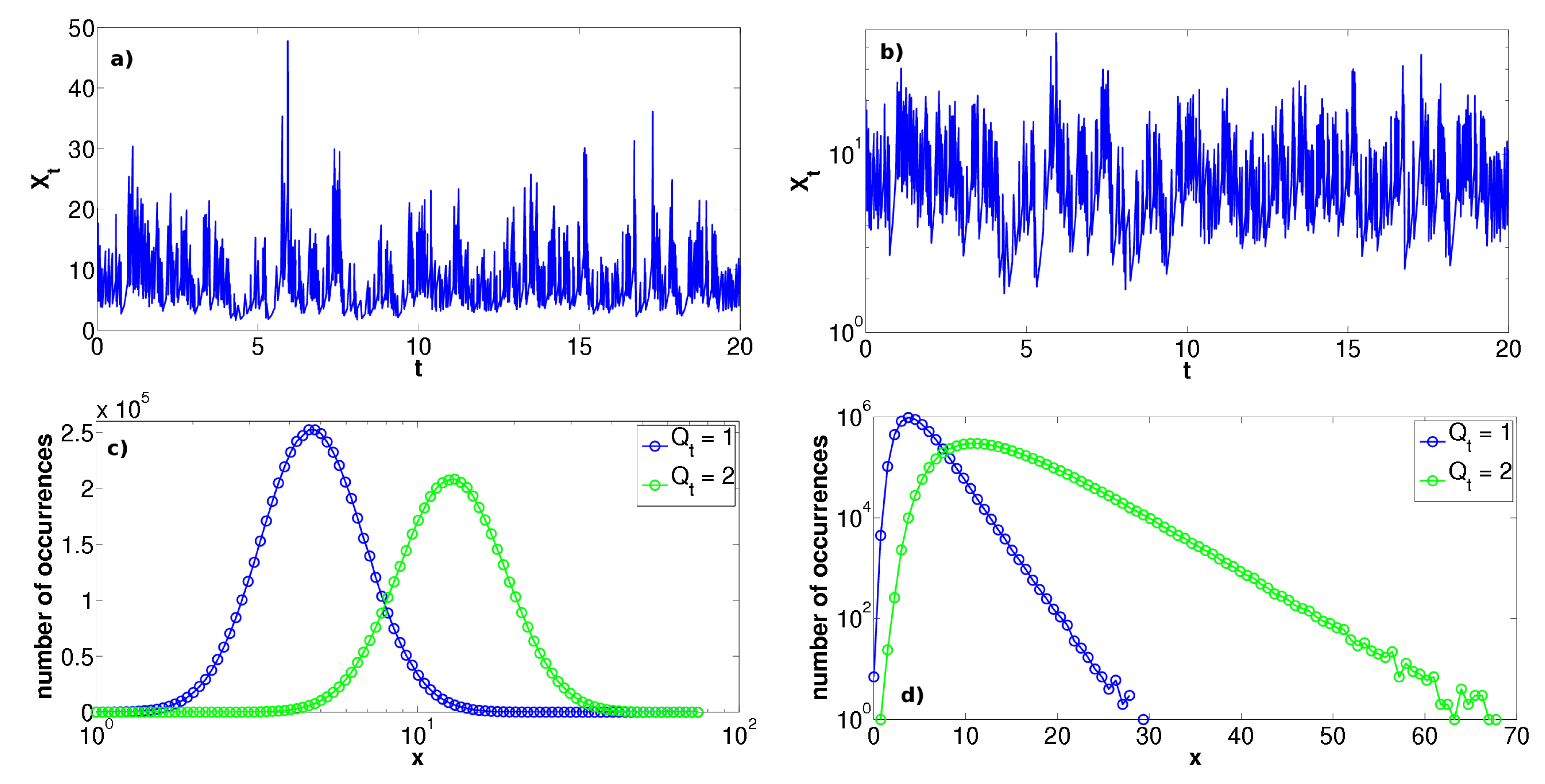}
  \caption{Two-state case that is moment stable, with exponential
    tails.  See text for description and explanation.}
  \label{fig:Two}
\end{centering}
\end{figure}

Here, we have chosen
\begin{equation*}
  f(q,x) = \alpha_q x^2,\quad \lambda_{q'}(q,x) = \beta_{q,q'} x^2, 
\end{equation*}
where $q,q'\in\{1,2\}$, and as always the resets are
$\gamma_{q,q'}x$.  We chose the $\a,\b,\g$ as follows:
\begin{equation*}
  \a = \left(\begin{array}{c} 2 \\ 1\end{array}\right),\quad
  \b = \left(\begin{array}{cc} 0&1 \\ 1&0\end{array}\right),\quad
  \g = \left(\begin{array}{cc} 0&2 \\ 1/3&0 \end{array}\right).
\end{equation*}
Since $\gamma_{12}\gamma_{21} < 1$, by the results above all moments
are stable, and this is what we observe numerically.

In Figures~\ref{fig:Two}a and ~\ref{fig:Two}b, we plot a single
trajectory of the system (in (a) we have plotted this in a
linear-linear scale, and in (b) we plot the same data in a linear-log
scale).  We have only plotted a subset of the entire realization here;
the full realization of $10^4$ jumps goes until approximately $t=254$,
and here we are only plotting $812$ jumps and cutting off at $t=20$ in
order to see more structure.

In Figure~\ref{fig:Two}c,d, we plot aggregate histograms of $X_t$,
each plot uses $0.5\times 10^6$ datapoints. Recall that each
realization of the process was run for $10^4$ jumps; we discarded the
first half of these for each realization, giving $0.5\times 10^3$
points, then aggregated across realizations.  The first observation is
that the distribution is bimodal, and this is due to the up-and-down
jumps: since $\gamma_{12}>1$ and $\gamma_{21}<1$, we expect the
typical value of $X_t$ to be much higher when $Q_t= 2$ than when it
equals one.  We separate out the data by the value of $Q_t$.  In
Figure~\ref{fig:Two}c, we plot on a log-linear scale, and it is pretty
apparent to the eye that the distributions for $X_t$, conditioned on
$Q_t$, are close to log-normal, as was the distribution in the
one-state case (q.v.~Figure~\ref{fig:One}).  We checked this
observations with QQ-plots (not presented here) and saw the same
phenomenon observed in Figure~\ref{fig:One}.  In (d), we plot a
histogram on a linear-log scale, and the data shows an exponential
tail, consistent with the prediction that all of the moments are
uniformly bounded above.

In Figure~\ref{fig:Three}, we plot the results of a Monte Carlo
simulation for a three-state SHS.  Again, we simulate $10^3$
realizations of this system, and each realization was integrated until
$10^4$ jumps had occurred.

\begin{figure}[ht]
\begin{centering}
  \includegraphics[width=0.97\textwidth]{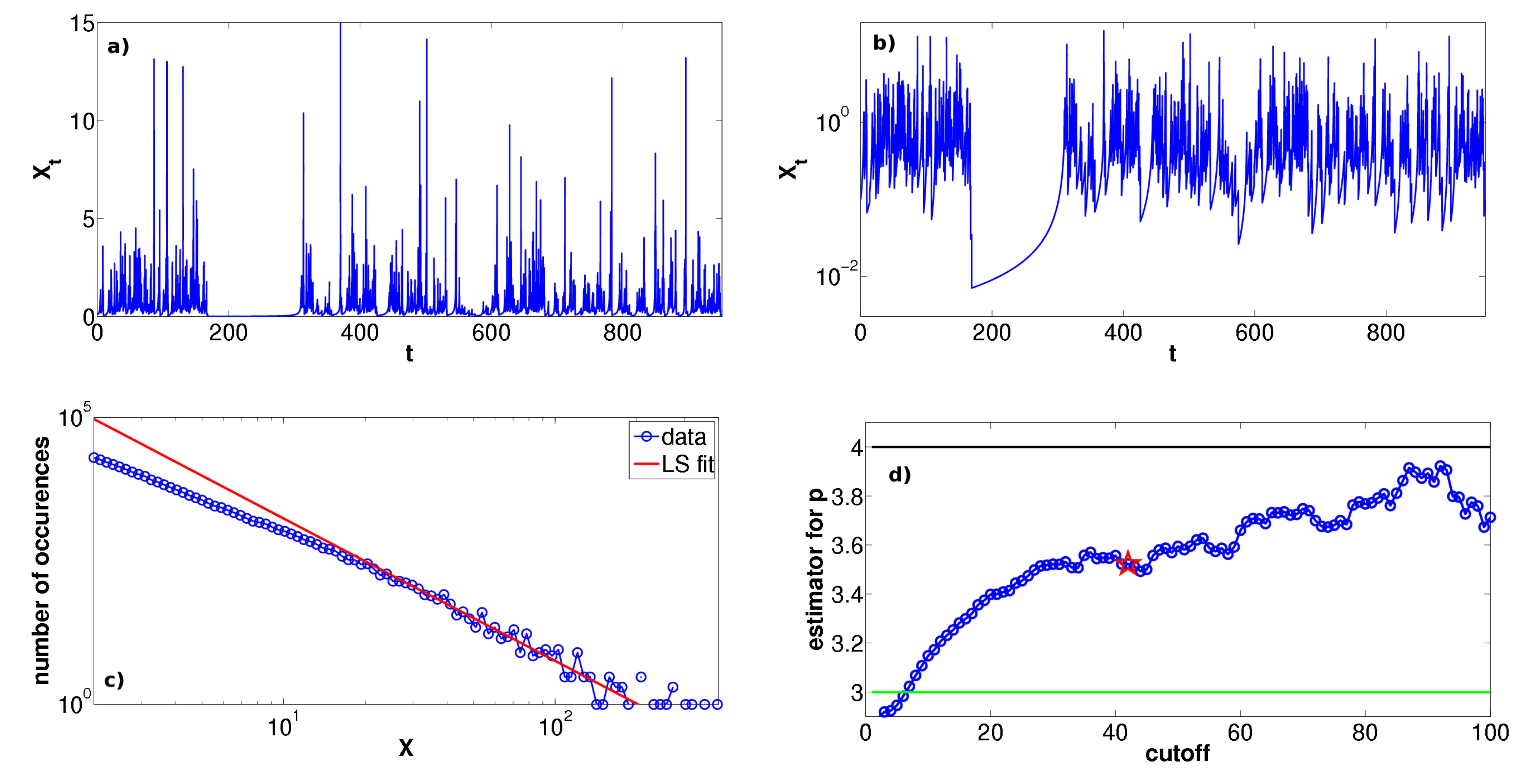}
  \caption{Three-state case that is marginally moment stable.  See
    text for description and explanation.}
  \label{fig:Three}
\end{centering}
\end{figure}

Here we have chosen
\begin{equation*}
  f(q,x) = \alpha_q x^2,\quad \lambda_{q'}(q,x) = \beta_{q,q'} x^2, 
\end{equation*}
where $q,q'\in\{1,2,3\}$, and as always the resets are
$\gamma_{q,q'}x$.  We chose the $\a,\b,\g$ as follows:
\begin{equation*}
  \a = \left(\begin{array}{c} 2 \\ 1 \\3\end{array}\right),\quad
  \b = \left(\begin{array}{ccc} 0&1&1 \\ 1&0&1 \\1&1&0\end{array}\right),\quad
  \g = \left(\begin{array}{ccc} 0&1.3&0.3 \\ 1.3&0&0.2 \\0.3&0.2&0\end{array}\right).
\end{equation*}
This means that 
\begin{equation*}
  M^{(m)} = \left(\begin{array}{ccc} -2 & (1.3)^m& (0.3)^m \\ (1.3)^m &-2 & (0.2)^m \\ (0.3)^m & (0.2)^m & -2\end{array}\right),
\end{equation*}
and one can see that 
\begin{equation*}
  (1.3)^2 + (0.3)^2 < 2,\quad (1.3)^3 > 2,
\end{equation*}
so by the theorems above, we have that the first two moments are
stable and the third is not  ---  thus as $t\to\infty$, $\E[X_t] <
\infty$ and $\E[X_t^2] < \infty$, but $\E[X_t^3]$ grows without bound.
Thus, if we see a power-law tail in the distribution, we expect it to
decay somewhere as $x^{-p}$ with $p$ between $3$ and $4$.

In Figure~\ref{fig:Three}(a,b), we again plot a single trajectory of the
system (in (a) we have plotted this in a linear-linear scale, and in
(b) we plot the same data in a linear-log scale).  We see the wide
variety of spatial scales in the linear plot, and the ``intermittent''
structure of the process in the log case  ---  sometimes the system can
be kicked quite low, and takes a while to climb out of it.  Recall
that the ODE is quadratically nonlinear, so if $X_t$ ever happens to
become small, it will take a long time to leave a neighborhood of
zero.

In Figure~\ref{fig:Three}c, we plot an aggregate histogram of
$0.5\times 10^6$ datapoints on a log-log scale.  Recall that each
realization of the process was run for $10^4$ jumps; we discarded the
first half of these for each realization, giving $0.5\times 10^3$
points, then aggregated across realizations.  We see from eye that it
looks to have a power law structure for large $x$, and we also plot
the least squares linear fit to the higher half of the data. This
slope is given as $p = -2.4797\pm 0.0112$, which seems to contradict
our assessment from above.  However, it should be pointed out that
trying to match a power law fit by least squares on a log-log plot
gives in general a bad estimator~\cite{Newman.05,
  Clauset.Shalizi.Newman.09}, so we reevaluate our analysis of the
decay.

We use the procedure laid out in~\cite{Clauset.Shalizi.Newman.09} as
follows: given a data set $\{x_i\}$, first decide a cutoff
$x_{\mathsf{min}}$, let $n = \#\{x_i > x_{\mathsf{min}}\}$, and this
gives the estimator
\begin{equation*}
  \hat{p} = 1 + \frac{n}{\sum_{i=1}^n \log(x_i/x_{\mathsf{min}})\1(x_i>x_{\mathsf{min}})}.
\end{equation*}
In Figure~\ref{fig:Three}d, we plot $\hat{p}$ for all cutoffs in the
range $[1,100]$, and we see that as long as we use a cutoff of about
10 or more, this is consistent with the theoretical prediction.  [The
horizontal lines in Figure~\ref{fig:Three}(d) are the bounds given by
the analysis.] We plot by a red star the value that is given by the
``best'' estimator, determined in the following manner: for each
choice of $x_{\mathsf{min}}$, we compare the empirical data with the
theoretical distribution assuming that our estimate of $\hat{p}$, then
measure the closeness of these distributions using the
Kolmogorov-Smirnoff (KS) distance.  The best fit was given by a cutoff
of $42$, with a estimator of $\hat{p} = 3.523$, and a KS distance
between the theoretical and empirical distributions of $7.78\times
10^{-3}$.

\section{Finite-time blowups}\label{sec:blowup}

In this section, we will again return to the one-state case considered
in Section~\ref{sec:mc1}.  It was shown there that as long as
$\dl\ge\df$, the stochastic process is ``well-behaved'', i.e. as
$t\to\infty$, all of the moments are finite.  In particular, it is not
possible for the process to escape to infinity in finite time.

However, in general, if the driving ODE is nonlinear, then it is
certainly possible that the process escapes to infinity in finite time
(these events are usually called {\bf finite-time blowups} in the
dynamical systems literature, or {\bf explosions} in the stochastic
process literature).  We examine these behaviors in this section.

\begin{definition}
  Let us denote by $T_n$ the times where the SHS has jumps.  We say
  that a realization of an SHS has a {\bf finite-time blowup} if either
  of two conditions hold:
\begin{enumerate}
\item[I.] $T_\infty:=\lim_{n\to\infty} T_n < \infty,$ 
\item[II.] the solution of the continuous part becomes infinite between any two jump times.
\end{enumerate}
We will refer to the two types of blowup as Type I and Type II.
\end{definition}

\begin{remark}
  What we are calling a ``Type I'' blowup is the type of behavior that
  is typically called an explosion for stochastic processes.  These
  are common for countable Markov chains where the jump rates can grow
  sufficiently fast.  What we call a ``Type II'' blowup is very much
  like what is called a finite-time blowup in the differential
  equations literature  ---  in this context, since the jump never
  occurs, it is the same as an ODE going to infinity.
\end{remark}

The results of this section can be summarized as follows: let $X_t =
\SHS(f,\lambda,\gamma)$ with $\df > 1$; then, we have the following two cases:

\begin{enumerate}
  
\item[C1.] If {$\dl < \df - 1$}, then $X_t$ will have a Type II blowup with
  probability one for any initial condition;
    
\item[C2.] if {$\dl = \df-1$}, then, depending on the leading coefficients
  of $f$ and $\lambda$, the system exhibits various behaviors,
  summarized in Theorem~\ref{thm:type1} below.  In particular, in some
  parameter regimes it exhibits Type I blowups almost surely, and the
  remaining regimes it does not exhibit Type I blowups almost
  surely. If $\df=1$, then the system cannot blowup, since linear
  flows are well-defined for all time.  However, as we show below, in
  the parameter range where the nonlinear systems exhibit finite-time
  blowups, the linear system still go to infinity but take an infinite
  amount of time to do so, exhibiting in some sense an ``infinite-time
  blowup''.

\end{enumerate}

We consider these two cases in the subsections below.

\subsection{Case C1: $\dl < \df - 1$}

We first consider the case where $\dl < \df-1$.  We show here that any
such system has a Type II blowup --- basically, the ODE goes to
infinity too fast for the jump rates to catch up.

\begin{theorem}\label{thm:type2}
  Let $X_t= \SHS(f,\lambda,\psi)$, and assume that $\dl < \df - 1$.
  Then $X_t$ has a Type II blowup with probability one for any initial
  condition.
\end{theorem}

\begin{proof}
  Using standard ODE arguments, if
  \begin{equation*}
    \frac{d}{dt}\varphi^t(x) = f(\varphi^t(x)),
  \end{equation*}
  where $f(x)$ is a positive polynomial of degree $\df>1$, then
  $\varphi^t(x)$ has a finite-time singularity at some $t^*<\infty$,
  and, moreover, in a (left) neighborhood of this point, we have
  \begin{equation*}
    \varphi^t(x) \sim C(t-t^*)^{1/(1-\df)}.
  \end{equation*}
  If $\dl<\df-1$, then $0> \dl/(1-\df)>-1$. Writing $\lambda(x) =
  \beta x^\dl + O(x^{\dl-1})$, we have that $\lambda(\varphi^s(x))$ is
  integrable near this singular point, i.e.
  \begin{equation*}
    \int_{t_0}^{t^*} \lambda(\varphi^s(x))\,ds  < \infty.
  \end{equation*}
  Since any exponential has support on the whole real axis, this
  and~\eqref{eq:defofT} means that starting at any initial condition,
  the probability of the ODE having a singularity before jumps is
  positive, i.e.
  \begin{equation*}
    \P\left(S_n  >  \int_{t_0}^{t^*} \beta(\varphi^s(x))\,ds\right) > 0,
  \end{equation*}
  and the system has a Type II blowup with positive probability.
  Since the process is Markov, irreducible on $\R^+$, and aperiodic,
  it follows that the system has a Type II blowup with probability
  one. \end{proof}

\subsection{Case~C2: $\dl = \df - 1$}

This case is much more complicated than the previous one for various
reasons, not the least of which being that the behavior depends on the
coefficients of the functions $f$ and $\lambda$.  The behavior is summarized in the following theorem:

\begin{theorem}\label{thm:type1}
Let $X_t = \SHS(f,\lambda,\gamma)$, where
\begin{equation}\label{eq:flasymp}
  f(x) = \alpha x^{k+1} + O(x^k),\quad \lambda(x) = \beta x^k + O(x^{k-1}),
\end{equation}
$k\ge1$ (so that the ODE is superlinear), and all of the coefficients
of both polynomials are positive. Then:

\begin{enumerate}

\item if $\gamma > e^{-\alpha/\beta}$, then $X_t$ has Type I blowups almost surely;
\item if $1-\alpha/\beta < \gamma < e^{-\alpha/\beta}$, then $X_t\to0$
  in probability, even though $\E[X_t]\to\infty$ as $t\to\infty$;
\item if $\gamma < 1-\alpha/\beta$, then $X_t\to0$, both in $L^1$ and almost surely.
\end{enumerate}
In particular, $X_t$ will not have Type II blowups.  

Finally, if we assume a linear ODE, i.e. $X_t = \SHS(\alpha x, \beta,
\gamma)$, then the conclusions above are all true except for \#1, in
which case we have that $X_t\to\infty$ a.s., but $X_t<\infty$ w.p.~1
for any finite $t$.

\end{theorem}

The way we proceed to prove this is as follows.  We first compute a
recurrence relation when $f,\lambda$ are assumed to be pure monomials,
and then show that this recurrence relation has growth properties that
correspond to the three types of convergence above, in the appropriate
parameter regimes.  Finally, we prove a monotonicity theorem that
allows us to extend the calculation for monomials to general
polynomials.

\begin{proposition}\label{prop:nottype2}
  If $X_t = \SHS(f,\lambda,\gamma)$ with $\dl = \df - 1$, then the
  probability of a Type II blowup is zero for any initial condition.
\end{proposition}
\begin{proof}
  Since $\dl = \df-1$, then $\lambda(\varphi^t(x))\sim 1/t$ in a
  neighborhood of the singularity (see the beginning of the proof of
  Proposition~\ref{thm:type2} for more detail), so that
  \begin{equation*}
    \int_{t_0}^{t^*} \lambda(\varphi^s(x))\,ds  = \infty,
  \end{equation*}
  and  a Type II blowup is not possible.
\end{proof}

\begin{lemma}\label{lem:recurrence}
  Let $X_t= \SHS(f,\lambda,\gamma)$ with $f(x) = \alpha x^{k+1}$ and
  $\lambda(x) = \beta x^k$.  Then
  \begin{equation*}
    X_{T_n} = \gamma e^{(\alpha/\beta) S_n}X_{T_{n-1}},
  \end{equation*}
  where $S_n$ are iid unit-rate exponential random variables.
\end{lemma}

\begin{proof}
  It is not hard to check that
  \begin{equation*}
    \varphi^t(x) = \frac{x}{(1-\alpha k t x^k)^{1/k}},
  \end{equation*}
  and thus
  \begin{equation*}
    \lambda(\varphi^t(x))  =\frac{\beta x^k}{1-\alpha k t x^k}.
  \end{equation*}
Therefore, we have
  \begin{align*}
    S_n 
    &= \int_0^{T_{n+1}-T_n} \lambda(\varphi^s(X_{T_n}))\,ds= \int_0^{T_{n+1}-T_n} \frac{\beta (X_{T_n})^k}{1-\alpha k s (X_{T_n})^k}\,ds\\
    &= \frac{-\beta}{\alpha k} \ln\av{1-\alpha k s (X_{T_n})^k}\Bigg|_{s=0}^{s=T_{n+1}-T_n} = \frac{-\beta}{\alpha k} \ln\av{1-\alpha k (T_{n+1}-T_n) (X_{T_n})^k}.
  \end{align*}
  From this, one can see that 
  \begin{equation}\label{eq:tn}
    1-\alpha k (T_{n+1}-T_n) (X_{T_n})^k = e^{-\alpha/(\beta k) S_n}.
  \end{equation}
  We then compute
\begin{equation}\label{eq:eab}
\begin{split}
    X_{T_{n+1}}
    &= \gamma \varphi^{T_{n+1}-T_n}(X_{T_n})\\
    &= \gamma \frac{ X_{T_n}}{(1-\alpha k (T_{n+1}-T_n) (X_{T_n})^k)^{1/k}} = \gamma e^{\frac{\alpha}{\beta}S_n}X_{T_n}.  
\end{split}
\end{equation}
\end{proof}

\begin{lemma}\label{lem:Zn}
  Let $S_k$ be iid unit rate exponentials and define
  \begin{equation*}
    Z_n = \prod_{k=1}^n \gamma e^{(\alpha/\beta)S_k}.
  \end{equation*}
  Then:
  \begin{enumerate}
  \item if $\gamma > e^{-\alpha/\beta}$, then there exists a
    $\delta'>0$ such that $\liminf_n e^{-\delta' n}Z_n = \infty$ with
    probability one;
  \item if $1-\alpha/\beta < \gamma < e^{-\alpha/\beta}$, then $Z_n\to0$
    in probability, even though $\E[Z_n]\to\infty$ as $t\to\infty$;
  \item if $\gamma < 1-\alpha/\beta$, then $Z_n\to0$, both in $L^1$ and almost surely.
  \end{enumerate}
\end{lemma}

\begin{proof}
  Let us write $W_k = \gamma e^{(\alpha/\beta)S_k}$.  Let us write $\mu
  = \E[W_k]$ and $\delta = \E[\log(W_k)]$.  We compute:
\begin{equation*}
  \mu = \gamma \E[e^{(\alpha/\beta)S_k}] = \gamma \int_0^t e^{(\alpha/\beta)t} e^{-t}\,dt = \frac{\gamma}{1-\alpha/\beta}.
\end{equation*}
We also have
\begin{equation*}
  \delta = \E\left[\frac\alpha\beta S_k + \log\gamma\right] = \frac\alpha\beta + \log\gamma.
\end{equation*}
Note then that the three conditions in this lemma correspond to
$\delta > 0$; $\delta < 0$ and $\mu >1$; and $\mu < 1$, respectively.
(By Jensen's inequality, we have
\begin{equation*}
  \mu = \E[W_k] = \E[e^{\log(W_k)}] > e^{\E[\log(W_k)]} = e^\delta,
\end{equation*}
so clearly $\delta> 0$ implies $\mu > 1$, and these are the only three
possibilities.)

First, we compute
\begin{equation*}
  \E[Z_n] = \E\left[\prod_{k=1}^n W_k\right] =  \prod_{k=1}^n \E[W_k] = \mu^n,
\end{equation*}
so clearly the expectation goes to zero (resp.~$\infty$) if $\mu<1$
(resp.~$\mu>1$).  Since $Z_n\ge 0$ by definition, this implies that
$Z_n\to0$ both in $L^1$ and almost surely.  This establishes claim \#3
of the lemma.

Next, note that
\begin{equation*}
  Z_n = \exp(\log(Z_n)) = \exp\left(\sum_{k=1}^n \log W_k\right) = \exp\left(\frac\alpha\beta \sum_{k=1}^n S_k + n \log\gamma\right).
\end{equation*}
From the Law of Large Numbers, 
\begin{equation*}
  \sum_{k=1}^n \log W_k = n \E[\log W_k] \pm O(\sqrt{n}) = n\left(\frac\alpha\beta + \log\gamma\right)\pm O(\sqrt{n}),
\end{equation*}
and if $\delta\neq 0$, this sum goes to $\pm\infty$ depending on the
sign of $\delta$.  More specifically, we can use Chernoff-type
bounds~\cite{SW,DZ}: let $\delta > 0$.  Since 
\begin{equation*}
  \E\left[\sum_{k=1}^n \log W_k\right] = n\delta,
\end{equation*}
the Chernoff bounds give
\begin{equation*}
  \P\left(\sum_{k=1}^n \log W_k < n\delta/2\right) < e^{-n\delta/8},
\end{equation*}
or
\begin{equation}\label{eq:ub}
  \P\left(Z_n < e^{n\delta/2}\right) < e^{-n\delta/8}.
\end{equation}
In particular, this means that $Z_n$ grows faster than an exponential
with probability exponentially close to one.  For example, choose $0 <
\delta' < \delta/2$.  Now, if it is not true that $\liminf e^{-\delta'
  n}Z_n = \infty$, this means that there exists $M>0$ and an infinite
sequence $n_1,n_2,\dots, n_k, \dots$ such that $e^{-\delta'
  {n_k}}Z_{n_k}\le M$.  From~\eqref{eq:ub}, this event has probability
zero.  This proves claim \#1 of the lemma.

Conversely, if $\delta < 0$, then
\begin{equation*}
  \P\left(Z_n > e^{n\delta/2}\right) < e^{-n\delta/12},
\end{equation*}
or the infinite product goes to zero exponentially fast with
probability exponentially close to one, which implies that $Z_n\to 0$
in probability.  Thus, if $\delta<0$ and $\mu>1$, then we have that
$Z_n\to0$ in probability but $\E[Z_n]\to\infty$, proving claim \#2 of
the theorem.
\end{proof}

\begin{lemma}\label{lem:monotone}
  Let $X_t = \SHS(f,\lambda,\gamma)$ and $Y_t = \SHS( g, \mu,
  \gamma)$, and assume that  
  \begin{equation}\label{eq:monotone}
    X_0 \le Y_0, \quad g(x)\ge f(x),\quad \mu(x) \le \lambda(x).
  \end{equation}
  If $X_t(\omega)$ has a finite-time blowup, then so does
  $Y_t(\omega)$.  Thus, the probability of $Y_t$ having a finite-time
  blowup is at least as large as $X_t$ having one, and so, for
  example, if $X_t$ has an a.s.  finite-time blowup, then so does
  $Y_t$.

  Conversely, if there exists an $M>0$ such that for all $x \le
  M$,~\eqref{eq:monotone} holds, and $X_0 \le Y_0 \le M$, if
  $Y_t(\omega)\to0$ as $t\to\infty$, then so does $X_t(\omega)$, and
  thus the probability of $X_t$ decaying to zero is at least as large
  as the probability that $Y_t$ does.
\end{lemma}

\begin{proof} 
  We denote $T_n^{(X)}$ as the $n$th reset time for process $X_t$, and
  similarly for $T_n^{(Y)}$.  Since $\mu \le \lambda$, $T_1^{(Y)} \ge
  T_1{(X)}$.  Together with the fact that $g \ge f$, this implies that
  $Y_{T_1^{(Y)}} \ge X_{T_1^{(X)}}$.  Using induction, the random sequence
  $Y_{T_n^{(Y)}}$ dominates the sequence $X_{T_n^{(X)}}$ for any $\omega$,
  and the conclusions follow.
\end{proof}

\begin{remark}
  Said in words: if we make the vector field larger, or the rate
  smaller, then the system is more likely to blow up.  Conversely, if
  we make the vector field smaller, or the rate larger, the system is
  more likely to go to zero.
\end{remark}

{\bf Proof of Theorem~\ref{thm:type1}.}  Let us first consider the
case where $f$ and $\lambda$ are pure monomials, i.e., $f(x) = \alpha
x^{k+1}$, $\lambda(x) = \beta x^k$.  Let us first consider the case
where $\gamma > e^{-\alpha/\beta}$.  Using Lemmas~\ref{lem:recurrence}
and~\ref{lem:Zn}, we know that there exists $\d'>0$ such that
$\liminf_n e^{-\delta' n}X_{T_n} = \infty$ with probability one.  This
means that $X_{T_n}$ is blowing up at least exponentially fast as a
function of $n$.  Moreover, using~\eqref{eq:tn}, we can solve
\begin{equation*}
  T_{n+1}-T_n = \frac{1-e^{-(\alpha/\beta k) S_n}}{\alpha k (X_{T_n})^k} \le C e^{-\delta' kn},
\end{equation*}
and if $k>1$, the telescoping sum is summable and $T_\infty < \infty$.
For the other two parameter regimes, the statement follows directly.

If $k=1$, then the recurrence in~\eqref{eq:tn} still holds, but now we
have $T_{n+1}-T_n = S_n/\beta$, and it is straightforward to see that
the telescoping sum is not summable w.p.~1, so that $T_\infty =
\infty$.

Now, for the general polynomial, we use Lemma~\ref{lem:monotone}.
Consider a general $f$ and $\lambda$, and assume that $\gamma >
e^{-\alpha/\beta}$.  If we choose $\beta'>\beta$ so that $\gamma >
e^{-\alpha/\beta'}$, then the argument above shows that $\SHS(\alpha
x^{k+1},\beta' x^k, \gamma)$ has Type I blowups a.s., and
Lemma~\ref{lem:monotone} implies that $\SHS(f,\lambda,\gamma)$ does as
well.  Similarly, if $\gamma < e^{-\alpha/\beta}$, choose
$\alpha'>\alpha$ with $\gamma < e^{-\alpha'/\beta}$.  Then the above
argument implies that $\SHS(\alpha' x^{k+1},\beta x^k, \gamma)\to0$
a.s., so Lemma~\ref{lem:monotone} implies that
$\SHS(f,\lambda,\gamma)\to0$ a.s. as well.
\qed

\subsection{Numerical results}

\newcommand{\tleft}{t_{\mathsf{left}}}
\newcommand{\tright}{t_{\mathsf{right}}}

In Figure~\ref{fig:Blowup} we plot the results of several Monte Carlo
simulations for blowups.  These simulations required a technique more
sophisticated than the Gillespie--Euler method used earlier. We are
attempting to simulate a system where nonlinear ODE are blowing up,
meaning that the vector fields get large and will be very sensitive to
discretization errors.  We implemented a two-phase method as follows:
if $X_t<10$, then we implemented a Euler--Gillespie method as
described before, with $\dtmax = 10^{-2}$. When $X_t\ge 10$, we
switched to a shooting method to obtain the next stopping time.  At
any $X_t$, we compute the time of singularity that would occur if
there were no jumps, call this $\tright$, and define $\tleft= 0$.
Choose $S_{n+1}$ as an exponential, and from this we can either
determine if we will have a Type II blowup before the next jump or
not.  If not, we then use a bisection method to find $T_{n+1}$: at any
stage in the process, we choose the midpoint of $[\tleft, \tright]$
and determine whether the integral in~\eqref{eq:defofT} is larger or
smaller than $S_{n+1}$; if smaller, we set $\tleft$ to be this
midpoint, and if larger, we set $\tright$ to be this midpoint.  This
guarantees an exponential convergence to $T_{n+1}$, and from this we
can compute all of the other quantities of interest.  Finally, we
always truncated whenever the system passed $10^9 \approx \exp(20.7)$
 ---  any time $X_t > 10^9$, we halted the computation.

In Figure~\ref{fig:Blowup}a, we plot $\log(X_t)$ for a single
realization of the SHS where we have chosen $f(x) = x^3$, $\lambda(x)
= 2 x^2$, and $\gamma = 0.75$.  One can see that there is a
finite-time blowup, and in fact we see that there are many jumps
happening in a very short time, as the trajectory goes off to
infinity.  This comports with the prediction of a Type I blowup.  In
Figure~\ref{fig:Blowup}b, we plot $\log(X_t)$ for a single realization
of the SHS where we have chosen $f(x) = x^4$, $\lambda(x) = 2 x^2$,
and $\gamma = 0.5$.  One can see that there is a finite-time blowup
here as well, but there are only a few jumps on the way to infinity;
this comports with the prediction of a Type II blowup.

In Figures~\ref{fig:Blowup}c and ~\ref{fig:Blowup}d, we present the results of a family of
simulations for Type I blowups.  Here we run each simulation for $10^4$
steps, or it has a finite-time blowup, and for each value of $\gamma$
we computed $10^2$ realizations. We chose $f(x) = x^3$, $\lambda(x) =
2x^2$ throughout, but vary $\gamma$ in the range $0.1,\dots,0.9$.
According to Theorem~\ref{thm:type1}, the critical values of $\gamma$
are $1/2$ and $e^{-1/2} \approx 0.6061$  ---  in both figures we have
put vertical red lines at these values.  

In both Figure~\ref{fig:Blowup}c and \ref{fig:Blowup}d we use the
plotting convention of plotting each simulation with a light blue
small circle, then plot the mean and standard deviation for all $10^2$
realizations for each gamma value in dark blue with a circle at the
mean and an error bar for the standard deviation.  In (c), we plot the
logarithm of the empirical mean of $X_t$ over the entire simulation,
recalling that we are truncating any trajectory that passes $10^9
\approx \exp(20.7)$.  Thus observations near or exceeding 20 are all
blowups.  We see that for $\gamma<1/2$, all realizations stay small
throughout the simulation.  In the range $[1/2, e^{-1/2}]$, there is a
spread of values depending on realization, and past $e^{-1/2}$ the
blowups dominate.  One can see this more starkly in
Figure~\ref{fig:Blowup}d: here we have plotted the final time of the
simulation --- note that we run all simulations for $10^4$ steps, and
the maximum threshold for the na\"{i}ve method is $\dtmax = 10^{-2}$
--- so if the trajectory always stays small, we would expect a total
simulation time very close to $10^4\cdot 10^{-2} = 100$.  Thus we can
interpret a final simulation time near $100$ as a proxy for a blowup
not occurring; conversely, if the simulation truncates significantly
earlier than $t=100$, this is a sign that the system has had a
finite-time blowup, and we see this clearly for large $\gamma$.

\begin{figure}[ht]
\begin{centering}
  \includegraphics[width=0.97\textwidth]{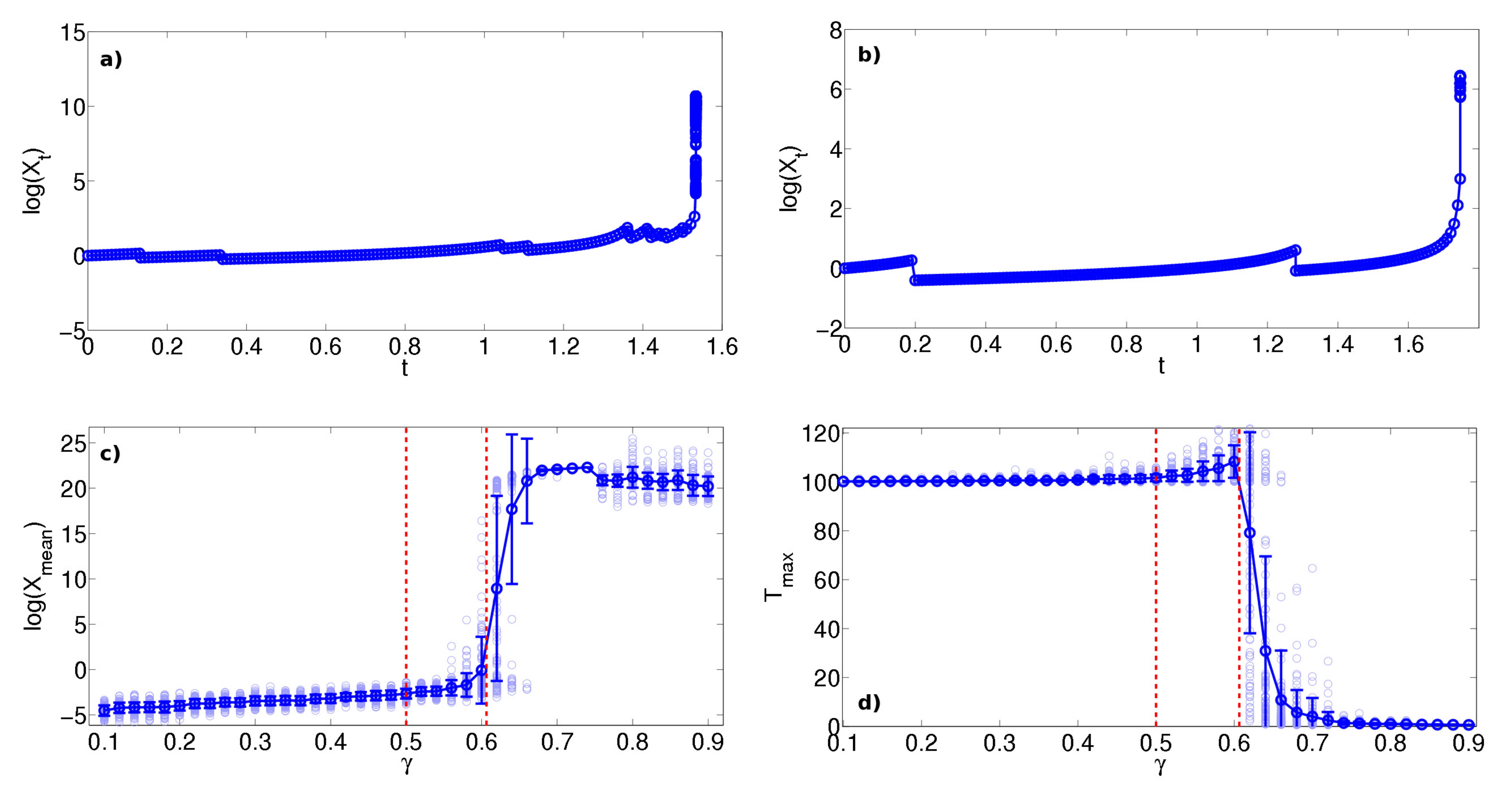}
  \caption{Results of simulations for SHS undergoing blowups; see text
    for more detailed description}
  \label{fig:Blowup}
\end{centering}
\end{figure}

\section{Exact computations} \label{sec:exact}

We found the recurrence relations~(\ref{eq:tn},~\ref{eq:eab}) useful
above in proving whether or not an SHS with certain parameters had
blowups or not.  In this section, we use recurrence relations to
compute exact invariant distributions for SHS, for certain choices of
$f$ and $\lambda$.

\subsection{General relation}

If we write $g(x) = \lambda(x) / f(x)$, and write $G'(x) = g(x)$, then
we compute:
\begin{align*}
  S_n 
  &= \int_0^{T_{n+1}-T_n} \lambda(\varphi^t(X_{T_n}))\,dt\\
  &= \int_0^{T_{n+1}-T_n} g(\varphi^t(X_{T_n}))f(\varphi^t(X_{T_n}))\,dt\\
  &= \int_0^{T_{n+1}-T_n} g(\varphi^t(X_{T_n}))\frac{d}{dt}\varphi^t(X_{T_n})\,dt\\
  &= G(\varphi^{T_{n+1}-T_n}(X_{T_n})) - G(X_{T_n}).
\end{align*}
Since $g(x) > 0$ for all $x$, $G(x)$ is increasing, so we can write
\begin{equation*}
  X_{T_{n+1}} = \gamma \varphi^{T_{n+1}-T_n}(X_{T_n}) = \gamma G^{-1}(S_n + G(X_{T_n})).
\end{equation*}

This identity holds true for any $f$, $\lambda$, but at the cost that
$g(x)$ might be a rational function and thus $G(x)$ could be quite
complicated.  We get a nice solution if $g(x)$ is a monomial: if we
assume that $g(x) = \delta x^p$, 
\begin{align*}
  (\varphi^{T_{n+1}-T_n}(X_{T_n}))^{p+1} - (X_{T_n})^{p+1} &= (p+1)\delta^{-1}S_n,\\
  \gamma^{p+1}(\varphi^{T_{n+1}-T_n}(X_{T_n}))^{p+1} &= \gamma^{p+1}((X_{T_n})^{p+1}+(p+1)\delta^{-1}S_n),\\
  X_{T_{n+1}}^{p+1}&= \gamma^{p+1}((X_{T_n})^{p+1}+(p+1)\delta^{-1}S_n).
\end{align*}
Writing $Y_{n} = (X_{T_n})^{p+1}$, $\zeta = \gamma^p$, and $\theta =
(p+1)/\delta$, gives the recursion
\begin{equation}\label{eq:Yrecursion}
  Y_{n+1} = (\zeta Y_n + \theta S_n).
\end{equation}
Note that $\zeta\in(0,1)$, since it is a positive power of $\gamma$,
and $\theta\in\R$.  We compute
\begin{equation*}
  Y_{n} = \zeta^n Y_0 + \theta\sum_{k=1}^n \zeta^k  S_k.
\end{equation*}
Writing the final sum as 
\begin{equation*}
  V_n = \theta\sum_{k=1}^n \zeta^k  S_k,
\end{equation*}
and assuming $X_0$ has finite moments of all orders, then for all
$m>0$,
\begin{equation*}
  \lim_{n\to\infty} \E[(Y_{n})^m] = \E[(V_n)^m],
\end{equation*}
and we can study the latter.  Similarly, as $n\to\infty$, the pdf for
$\zeta^n Y_0$ converges to $\delta(0)$, so
\begin{equation*}
  p_{Y_{n}}(z) = p_{\zeta^n Y_0 + V_n}(z) = p_{\zeta^n Y_0}(z) * p_{V_n}(z) \to p_{V_n}(z),
\end{equation*}
and thus it is sufficient to study $V_n$ whether we are interested in
particular moments, or a formula for its distribution.  We now write
down formulas for the moments and distribution of $V_n$.

\subsection{Moments of $V_n$}

\begin{theorem}\label{thm:moment}
  The $m$th moment of $V_n$ is given by
\begin{equation}\label{eq:Znmoment}
  \E[(V_n)^m]=\theta^m \cdot m! \sum_{k=m}^{nm} \pi_{n}^{(k,m)} \zeta^k,
\end{equation}
where $\pi_n^{(k,m)}$ is the number of partitions of $k$ into $m$
positive parts, each of size less than or equal to $n$.
\end{theorem}

\begin{proof}
  Wlog assume $\theta = 1$ by rescaling.  We
  use the standard ``generatingfunctionology'' approach here: we write
\begin{equation*}
  \E[e^{t V_n}] = \E[e^{t\sum_{k=1}^n \zeta^k S_k}] = \E[\prod_{k=1}^n e^{t\zeta^k S_k}] = \prod_{k=1}^n\E[e^{t\zeta^k S_k}]  = \prod_{k=1}^n \frac{1}{1-t \zeta^k},
\end{equation*}
where we have used independence of the $S_k$, and this formula is
valid if $\av t < \zeta^{-k}$.  We also have
\begin{equation*}
  \E[e^{tV_n}] = \E\left[\sum_{m=0}^\infty \frac{(tV_n)^m}{m!}\right] = \sum_{m=0}^\infty \frac{t^m}{m!}\E[(V_n)^m].
\end{equation*}
Thus the moment of interest is $m!$ times the coefficient of $t^m$ in
the power series of $\E[e^{tV_n}]$.

We have
\begin{align*}
  \prod_{k=1}^n \frac{1}{1-t \zeta^k}
  &= \prod_{k=1}^n \sum_{\ell=0}^\infty (t \zeta^k)^\ell= \sum_{\ell_1,\ell_2,\dots,\ell_n=0}^\infty \prod_{k=1}^n(t \zeta^k)^{\ell_k}\\
  &= \sum_{\ell_1,\ell_2,\dots,\ell_n=0}^\infty t^{\sum_{k=1}^n \ell_k} \zeta^{\sum_{k=1}^n k\ell_k}\\
  &= \sum_{m=0}^\infty t^m \sum_{\substack{\ell_1,\ell_2,\dots,\ell_n=0\\\ell_1+\dots+\ell_n = m}} \zeta^{\sum_{k=1}^n k\ell_k}.
\end{align*}
Therefore
\begin{equation*}
  \E[(V_n)^m] = m!\sum_{\substack{\ell_1,\ell_2,\dots,\ell_n=0\\\ell_1+\dots+\ell_n = m}} \zeta^{\sum_{k=1}^n k\ell_k}
\end{equation*}
The coefficient of $\zeta^k$ in this sum is given by the number
of integer solutions to the Diophantine system
\begin{equation}\label{eq:Diophantine}
\begin{split}
  \ell_1+\ell_2+\dots+\ell_n &= m,\\
  \ell_1+2\ell_2+\dots+n\ell_n &= k.
\end{split}
\end{equation}

This number is equal to the number of partitions of $k$ into $m$
positive parts, each less than $n$.  To see this, consider such a
partition, ordered increasingly, so that it contains $\ell_1$ $1$'s,
then $\ell_2$ 2's, all up to $\ell_n$ $n$'s.  Then by definition, the
$\ell_j$ must satisfy~\eqref{eq:Diophantine}.  Conversely, any choice
of $\ell_j$ satisfying~\eqref{eq:Diophantine} gives a partition in the
obvious manner.  This completes the proof.
\end{proof}

\begin{corollary}\label{cor:partition}
\begin{equation*}
  \E[(V_\infty)^m] = \theta^m \cdot m! \sum_{k=m}^\infty \pi_m(k)\zeta^k,
\end{equation*}
where  $\pi_m(k)$ is the standard partition function of
$k$ into $m$ positive parts~\cite{A008284}.
\end{corollary}

\begin{remark}
  It is not entirely surprising that partition numbers appear in this
  computation.  Again, let $\theta= 1$ for simplicity.  From
  Corollary~\ref{cor:partition} that the generating function of
  $V_\infty$ is the infinite product
  \begin{equation*}
    \E[e^{t V_\infty}] = \prod_{\ell=1}^\infty \frac{1}{1-t\zeta^\ell}.
  \end{equation*}
  But we have
  \begin{equation*}
    \E[e^{t V_\infty}] = \sum_{m=0}^\infty \frac{t^m}{m!} \E[(V_\infty)^m] = \sum_{m=0}^\infty \sum_{k=m}^\infty \pi_m(k) t^m \zeta^k= \sum_{k=0}^\infty \zeta^k \sum_{m=0}^{k-1} t^m\pi_m(k).
  \end{equation*}
  Clearly $\sum_{m=0}^{k-1} \pi_m(k)$ is the standard partition
  function $\pi(k)$, and setting $t=1$ we obtain
\begin{equation*}
  \prod_{\ell=1}^\infty \frac{1}{1-\zeta^\ell} = \sum_{k=0}^\infty \zeta^k\pi(k),
\end{equation*}
recovering the well-known formula for the partition function.

\end{remark}

\subsection{Distribution of $V_n$}

\begin{theorem}
  The probability distribution function of $V_n$ is given by
  \begin{equation*}
    p_{n}(z) = \sum_{k=1}^n A_{k,n}(\zeta) e^{-z/\zeta^k}\1(z>0),
  \end{equation*}
  with 
  \begin{equation*}
    A_{k,n}(\zeta) = \frac{(-1)^{k-1} \zeta^{k(k-3)/2}}{\displaystyle\prod_{\substack{j=1\\j\neq k}}^n (1-\zeta^{\av{j-k}})}.
  \end{equation*}
\end{theorem}

\begin{proof}
  We prove this by induction.  Starting with $n=1$, we see that
  $A_{11}(\zeta) = 1/\zeta$. Since $V_1 = \zeta S_1$, and therefore
  \begin{equation*}
    \P(V_1 > z) = \P(\zeta S_1 > z)= \P(S_1 > z/\zeta) = e^{-z/\zeta}\1(z>0),\quad p_{1}(z) = -\frac{d}{dz}\P(V_1 > z) =\frac1\zeta e^{-z/\zeta}\1(z>0).
  \end{equation*}
  This proves the formula for $p_1(z)$.  Since $V_{n+1} = V_{n} + \zeta^{n+1}
S_{n+1}$, we have
\begin{align*}
  p_{n+1}(z) 
  &= \int_{-\infty}^\infty p_n(z-x) \zeta^{-(n+1)} e^{-x/\zeta^{n+1}}\1(x>0)\,dx\\
  &= \int_{-\infty}^\infty \sum_{k=1}^n A_{k,n}(\zeta) e^{-(z-x)/\zeta^k}\1(z-x>0)\zeta^{-(n+1)} e^{-x/\zeta^{n+1}}\1(x>0)\,dx\\
  &= \sum_{k=1}^n A_{k,n}(\zeta)e^{-z/\zeta^k} \int_0^z \zeta^{-(n+1)}\exp\left(\frac{\zeta^{n-k+1}-1}{\zeta^{n+1}}x\right)\,dx\\
  &= \sum_{k=1}^n A_{k,n}(\zeta)e^{-z/\zeta^k} \frac1{\zeta^{n-k+1}-1}\left(\exp\left(\frac{\zeta^{n-k+1}-1}{\zeta^{n+1}}z\right)-1\right)\\
  &= \sum_{k=1}^n A_{k,n}(\zeta)\frac{e^{-z/\zeta^{n+1}} - e^{-z/\zeta^k}}{\zeta^{n-k+1}-1}\\
  &= \sum_{k=1}^n \frac{A_{k,n}(\zeta)}{1-\zeta^{n-k+1}}e^{-z/\zeta^k} + e^{z/\zeta^{n+1}}\sum_{k=1}^n \frac{A_{k,n}(\zeta)}{\zeta^{n-k+1}-1}.
\end{align*}
This gives us the recursion
\begin{equation*}
  A_{k,n+1} = \frac{A_{k,n}}{1-\zeta^{n-k+1}},\quad k < n+1;\quad A_{n+1,n+1} = \sum_{k=1}^{n}\frac{A_{k,n}}{\zeta^{n-k+1}-1}.
\end{equation*}
By plugging in, we see that the formula for $A_{k,n}$ satisfies the
recursion relation, and we are done.
\end{proof}

\begin{corollary}
  The limit 
\begin{equation*}
  A_k(\zeta) := \lim_{n\to\infty} A_{k,n}(\zeta)
\end{equation*}
exists for all $\zeta\in(0,1)$ (and in fact  convergences 
uniformly), so that the distribution for $V_\infty$ is
\begin{equation*}
  p_{V_\infty}(z) = \sum_{k=1}^\infty A_k(\zeta) e^{-z/\zeta^k}.
\end{equation*}
\end{corollary}

\begin{remark}
We have the following formulas for $A_k(\zeta)$:
\begin{equation*}
  A_1(\zeta) = \frac1\zeta \frac1{\prod_{k=1}^\infty (1-\zeta^k)} = \frac1\zeta\sum_{k=1}^\infty \sum_{j=1}^\infty \zeta^{jk} = \frac1\zeta\sum_{n=1}^\infty \varphi(n) \zeta^n,
\end{equation*}
where $\varphi(n)$ is the number of distinct factors of the integer
$n$.  For $k>1$, we have that
\begin{equation*}
  \frac{A_k(\zeta)}{A_1(\zeta)} =  \frac{(-1)^{k-1}\zeta^{\frac{(k-2)(k-1)}2}}{\prod_{j=1}^{k-1} (1-\zeta^j)}.
\end{equation*}
Moreover, for any fixed $\zeta\in(0,1)$,
\begin{equation*}
  \lim_{k\to\infty} A_k(\zeta) = 0,
\end{equation*}
and in fact, for fixed $\zeta$ and $k$ sufficiently large,
\begin{equation*}
  \frac{A_k(\zeta)}{A_{k-1}(\zeta)} \approx \zeta^{k-2},
\end{equation*}
so that $A_k(\zeta)$ goes to zero superexponentially fast as $k\to\infty$.
\end{remark}

\section{Conclusions}\label{sec:outtro}

We have presented many results for a class of SHS and examined the
tension between instability arising from the flows and stability
arising from the resets.  The types of dynamical phenomena that we
have observed share many features with both deterministic dynamical
systems and stochastic processes.

One common feature of all of the systems studied here was that the
state space was one-dimensional.  Extending the types of results in
this paper to higher dimensions could be challenging due to
non-commutativity of flows.

\section{Acknowledgments}

This work was supported in part by the National Science Foundation
(NSF) under awards ECCS-CAR-0954420, CMG-0934491, and
CyberSEES-1442686, the Trustworthy Cyber Infrastructure for the Power
Grid (TCIPG) under US Department of Energy Award DE-OE0000097, by the
National Aeronautics and Space Administration through the NASA
Astrobiology Institute under Cooperative Agreement No. NNA13AA91A
issued through the Science Mission Directorate, and the Initiative for
Mathematical Science and Engineering (IMSE) at the University of
Illinois.

\end{document}